\documentclass[12pt,amstex]{amsart}
\usepackage{txfonts}

\usepackage{epsfig}
\usepackage{amsmath}
\usepackage{amssymb}
\usepackage{amscd}
\usepackage{graphicx}
\usepackage[all]{xy}
\usepackage{pstricks}
\usepackage{pst-node}
\usepackage[english]{babel}
\usepackage{float}
\usepackage{tikz}
\usetikzlibrary{matrix}
\newtheorem*{thm*}{Theorem}
\newtheorem{thm}{Theorem}[section]
\newtheorem{lem}[thm]{Lemma}

\newtheorem{prop}[thm]{Proposition}

\newtheorem{cor}[thm]{Corollary}

\theoremstyle{definition}
\newtheorem{defn}[thm]{Definition}

\theoremstyle{remark}
\newtheorem{rem}[thm]{Remark}

\numberwithin{equation}{section}

\def \Q {\textbf{Q}}
\def \wh {\widehat}
\def \to {\rightarrow}

\def \id {\text{id}}
\def \d {\delta}
\def \l {\lambda}

\def \I {\mathcal{I}}

\def \W {\I(S)\vee\I(T)}
\def \Z {\mathcal{Z}}
\newcommand*\normm[1]{\left|\mspace{-1mu}\left|\mspace{-1mu}\left|#1\right|\mspace{-1mu}\right|\mspace{-1mu}\right| }

\begin{document}
\title
[Pointwise averages for systems with two commuting transformations]{Pointwise Multiple averages for systems with two commuting transformations}
\author{Sebasti\'an Donoso}
\address{Center for Mathematical Modeling, University of Chile, Beauchef 851, Santiago, Chile.} \email{sdonoso@dim.uchile.cl}

\author{Wenbo Sun}
\address{Department of Mathematics, Northwestern University, 2033
Sheridan Road Evanston, IL 60208-2730, USA}
 \email{swenbo@math.northwestern.edu}


\thanks{The first author is supported by grants Basal-CMM. The second author is partially supported by NSF grant 1200971. The second author thanks the hospitality of University of Chile where this research was started, and the first author thanks the hospitality of Northwestern University where this research was finished.}

\begin{abstract}

We show that for every ergodic measure preserving system $(X,\mathcal{X},\mu,S,T)$ with commuting transformations $S$ and $T$, the average 

\[\frac{1}{N^3} \sum_{i,j,k=0}^{N-1} f_0(S^{j}T^{k}x)f_1(S^{i+j}T^{k}x)f_2(S^{j}T^{i+k}x)\] 
converges for $\mu$-a.e. $x\in X$ as $N\to \infty$ for all $f_0,f_1, f_2\in L^{\infty}(\mu)$. We also show that if $(X,\mathcal{X},\mu,S,T)$ is an ergodic measurable distal system, then the average \[ \frac{1}{N}\sum_{i=0}^{N-1} f_1(S^ix)f_2(T^ix) \] 
converges for $\mu$-a.e. $x\in X$ as $N\to \infty$ for all $f_1,f_2\in L^{\infty}(\mu)$.
\end{abstract}

\maketitle

\section{Introduction}
The convergence of multiple ergodic averages is a widely studied question in ergodic theory. The question is to know whether the average 
\begin{equation}\label{11111}
 \frac{1}{N}\sum_{i=0}^{N-1} f_1(T_1^ix)f_2(T_2^ix)\cdots f_d(T_d^i x)
 \end{equation}
converges as $N\to \infty$ for bounded functions $f_1,\ldots,f_d$, where $(X,\mathcal{X},\mu)$ is a probability space and $T_1,\ldots,T_d$ are measure preserving transformations of $X$ (we refer to $(X,\mathcal{X},\mu,$ $T_1,\ldots,T_d)$ as a \emph{system}). In the $L^2$ setting, this problem has a long history and satisfactory answers have been given up to now \cite{Aus,HK05,H,Tao}. The first breakthrough was done by Host and Kra \cite{HK05}, where they derived the $L^2$ convergence of
\begin{equation} \label{HKmultiple}
 \frac{1}{N}\sum_{i=0}^{N-1} f_1(T^ix)f_2(T^{2i}x)\cdots f_d(T^{di} x)
 \end{equation}
 as a consequence of a celebrated structure theorem for measure preserving systems with a single transformation.
The most general result was given by Walsh \cite{Walsh}, where he proved that (\ref{11111}) (and more general expressions) converges in the $L^2$ setting when $T_1,\ldots,T_d$ span a nilpotent group. 

In the pointwise setting, the situation is completely different: very few results are known. The most remarkable ones are those by Bourgain \cite{Bo}, where he proved the pointwise convergence of 
$ \frac{1}{N} \sum_{i=0}^{N-1} f_1(T^{ai}x)f_2(T^{bi}x)$ ($a,b\in \mathbb{Z}$), and that of Huang, Shao and Ye \cite{HSY}, who proved the convergence for the average \eqref{HKmultiple} in a measurable distal system. Their proof presents an original application of topological models to prove pointwise convergence results. 

In this article, we push forward this technique to the case of two commuting transformations, continuing the program started in \cite{DS}. We prove:

\begin{thm} \label{THM:1}
Let $(X,\mathcal{X},\mu,S,T)$ be an ergodic system with commuting transformations $S$ and $T$ ({ i.e.} $ST=TS$). Then the average
\[\frac{1}{N^3} \sum_{i,j,k=0}^{N-1} f_0(S^jT^kx)f_1(S^{i+j}T^k x)f_2(S^j T^{i+k}x) \]
converges for $\mu$-a.e. $x\in X$ as $N\to \infty$ for all $f_0,f_1,f_2\in L^{\infty}(\mu)$.
\end{thm}

In the distal case (see Section \ref{Sec:MultipleAverage} for definitions), we have: 
\begin{thm} \label{THM:2}
Let $(X,\mathcal{X},\mu,S,T)$ be an ergodic distal system with commuting transformations $S$ and $T$. Then the average
\[\frac{1}{N} \sum_{i=0}^{N-1} f_1(S^{i}x)f_2(T^{i}x) \]
converges for $\mu$-a.e. $x\in X$ as $N\to \infty$ for all  $f_1,f_2\in L^{\infty}(\mu)$.
\end{thm}
  
The construction of a suitable topological model is essential in proving these theorems. A {\em topological model} for an ergodic system $(X,\mathcal{X},\mu,S,T)$ is a topological dynamical system with a probability measure for which the systems are measure theoretical isomorphic. The importance of a topological model is that its algebra of continuous functions naturally provides a dense algebra of functions (in $L^1$ norm for example) to work with. The strategy is to require additional properties to the model such that this algebra satisfies suitable properties related to multiple averages.

In this paper, we introduce a topological structure $N_{S,T}(X)$ (see Section \ref{Sec:BuildTopModel} for the definition) and prove in Section 3 that passing to a suitable extension, every system $X$ has a topological model whose $N_{S,T}(X)$ structure is strictly ergodic (see Section \ref{Sec:Background} for definitions). We then use this model to deduce 
Theorems \ref{THM:1} and \ref{THM:2} in Section 4.

\section*{Acknowledgments} We thank Bernard Host, Bryna Kra and Alejandro Maass for useful and clarifying discussions. We also thank the referee for valuable comments and pointing out some unclear statements in an early version of this paper.
 
%In our setting we require the unique ergodicity of the following topological structure.
%Let $X$ be a compact metric space and $S\colon X\to X$ and $T\colon X\to X$ be two commuting homeomorphisms. 
%We define $N_{S,T}(X)$ to be the closure in $X^3$ of the points $(x,S^ix,T^ix)$ where $x\in X$ and $i\in \mathbb{Z}$. This set is invariant under the action spanned by $\id\times S\times T$, $S\times S\times S$ and $T\times T\times T$. The unique ergodicity of this structure under the transformations mentioned before (see Section \ref{Sec:TopStructures} for a precise statement) implies some properties of the the algebra of continuous functions of $X$ which allow us to deduce Theorems \ref{THM:1} and \ref{THM:2}. 

\section{Background and notation} \label{Sec:Background}
\subsection{Measure theoretic and topological systems}
A \emph{measure preserving system} is a tuple $(X,\mathcal{X},\mu,G)$, where $(X,\mathcal{X},\mu)$ is a probability space and $G$ is a group of measurable, measure preserving transformations acting on $X$. It is \emph{ergodic} if all $G$-invariant sets have measure either 0 or 1. We omit writing the $\sigma$-algebra $\mathcal{X}$ when there is no possible confusion. 
 
A measure preserving system $(X,\mathcal{X},\mu,G)$ is {\em free} (or the action of $G$ on $(X,\mathcal{X},\mu)$ is free) if all elements different from the identity have no fixed points, {\em i.e. } $\mu(\{x: gx=x\})=0$ for all $g\neq \id$.   

Given two $\sigma$-algebras $\mathcal{A}$ and $\mathcal{B}$, $\mathcal{A}\vee \mathcal{B}$ denotes the $\sigma$-algebra generated by $\{A\cap B: A\in \mathcal{A}, B \in \mathcal{B}\}$. It is the smallest $\sigma$-algebra containing $\mathcal{A}$ and $\mathcal{B}$. If $f$ is a bounded function on $X$ and $\mathcal{A}$ is a $\sigma$-algebra, $\mathbb{E}(f\vert\mathcal{A})$ denotes the conditional expectation of $f$ with respect to $\mathcal{A}$. 

A \emph{factor map} $\pi\colon X\to Y$ between the measure preserving systems $(X,\mathcal{X},\mu,G)$ and $(Y,\mathcal{Y},\nu,G)$ is a measurable function  such that $\pi_{*} \mu =\nu$ and $\pi \circ g=g\circ \pi$ for every $g\in G$ (here $\pi_{*}\mu(A)\coloneqq\mu(\pi^{-1}(A))$, $A\in \mathcal{Y}$ is the pushforward measure of $\mu$, and in a slight abuse of notation, $G$ denotes the group action on both $X$ and $Y$). In this case we say that $(Y,\mathcal{Y},\nu,G)$ is a \emph{factor} of $(X,\mathcal{X},\mu,G)$ and  $(X,\mathcal{X},\mu,G)$ is an \emph{extension} of $(Y,\mathcal{Y},\nu,G)$. It is worth noting that $\mathcal{Y}$ can be viewed as an invariant sub $\sigma$-algebra of $\mathcal{X}$ by identifying $\mathcal{Y}$ with $\pi^{-1}(\mathcal{Y})$. If $\pi$ is bijective (modulo null sets), we say that $\pi$ is an  {\em isomorphism} and that $(X,\mathcal{X},\mu,G)$ and $(Y,\mathcal{Y},\nu,G)$ are {\em isomorphic}.

Given a factor map $\pi\colon X\to Y$ between the measure preserving systems $(X,\mathcal{X},\mu,G)$ and $(Y,\mathcal{Y},\nu,G)$ and a function $f\in L^2(\mu)$, the {\em conditional expectation of $f$ with respect to $Y$} is the function $\mathbb{E}(f\vert Y)\in L^2(\nu)$ such that $\mathbb{E}(f\vert Y)\circ \pi=\mathbb{E}(f\vert \mathcal{Y})$ (we regard $\mathcal{Y}$ as a sub $\sigma$-algebra of $\mathcal{X}$). This expectation is characterized by the equation (see for example \cite{Fu}, Chapter 5) \[\int_Y \mathbb{E}(f\vert Y)\cdot g d\nu=\int_X f \cdot g\circ\pi d\mu \quad \text{ for every } g \in L^2(\nu). \]
There exists a unique measurable map $Y\to M(X)$, $y\mapsto \mu_y$ such that
$\mathbb{E}(f\vert Y)(y)=\int f d\mu_y$ for every $f\in L^1(\mu)$. The expression $\mu=\int_Y \mu_y d\nu(y)$ is called the {\it disintegration} of $\mu$ over $\nu$.

\

A \emph{topological dynamical system} is a pair $(X,G)$, where $X$ is a compact metric space and $G$ is a group of homeomorphisms of the space $X$. $(X,G)$ is \emph{minimal} if for any $x\in X$, its orbit $\{gx: g\in G\}$ is dense in $X$. $(X,G)$ is \emph{strictly ergodic} if it is minimal and its convex set of invariant measures consists of just one measure. 
A topological factor map is an onto continuous function $\pi\colon X\to Y$ such that $\pi \circ g=g\circ \pi$ for every $g\in G$. 

Usually we write $(X,\mathcal{X},\mu,T_1,\ldots,T_d)$ to denote that $T_1,\ldots,T_d$ span a group of measurable measure preserving transformations on $X$ (we adapt the same convention in the topological context) and sometimes we write a subscript to the transformations (like $S_X$ or $T_X$) to stress the space where they are acting. 

{\bf Convention: } When there is no confusion, if $(X,\mathcal{X},\mu,S,T)$ is a system with commuting transformations $S$ and $T$, we always write $R=S^{-1}T$ and $G=\langle S,T\rangle=\langle S,R\rangle=\langle T,R\rangle$ for the group spanned by $S$ and $T$. We add some subscripts to avoid confusion when several systems are involved. We also use this convention in the topological context.

\subsection{Relative Jewett-Krieger Theorem}
Let $(X,\mathcal{X},\mu,G_{0})$ be a measure preserving system. A {\em strictly ergodic model} for $(X,\mathcal{X},\mu,G_{0})$ is a strictly ergodic topological dynamical system $(\widehat{X},G_{0})$ which is measurable isomorphic to $(X,\mathcal{X},\mu,G_{0})$ when we endowed it with its unique invariant measure $\widehat{\mu}$. We usually use $\widehat{\cdot}$ to denote a topological model for a system. 

When the acting group is the integers $\mathbb{Z}$, the well-known Jewett-Krieger Theorem \cite{J,K} states that every ergodic measure preserving system has a strictly ergodic model. Weiss \cite{W} generalized this result to abelian group actions and gave a relative version of it, which is a fundamental tool we use in this article. 
\begin{thm}[Weiss, \cite{W}] \label{Weiss}
Let $G_0$ be an abelian group and $\pi\colon(X,\mathcal{X},\mu,G_{0})\to (Y,\mathcal{Y},\nu,G_{0})$ be a factor map between ergodic and free systems. Let  $(\wh{Y},G_{0})$ be a strictly ergodic model for $(Y,\mathcal{Y},\nu,G_{0})$. Then there exist a strictly ergodic model $(\wh{X}, G_{0})$ for $(X,\mathcal{X},\mu,G_{0})$ and a topological factor map $\wh{\pi}:\wh{X}\to \wh{Y}$ such that the following diagram commutes:

\begin{figure}[h]
 \begin{tikzpicture}
  \matrix (m) [matrix of math nodes,row sep=3em,column sep=4em,minimum width=2em,ampersand replacement=\&]
  {
     X \& \widehat{X} \\
     Y \& \widehat{Y} \\};
  \path[-stealth]
     (m-1-1) edge node [left] {$\pi$} (m-2-1)
    (m-1-1) edge node [above] {$\Phi$} (m-1-2)
    (m-1-2) edge (m-1-1)
    (m-2-2) edge (m-2-1)
    (m-1-2) edge node [right] {$\widehat{\pi}$} (m-2-2)
    (m-2-1) edge node [below] {$\phi$} (m-2-2);
\end{tikzpicture}
 \end{figure}
where $\Phi$ and $\phi$ are measure preserving isomorphisms such that $\pi\circ \Phi=\phi \circ \widehat{\pi}$.
\end{thm}
We refer to $\wh{\pi}\colon \wh{X}\to \wh{Y}$ as a {\it topological model} for $\pi\colon X\to Y$.

\subsection{Facts about the $\Z_{W_1,W_2}$ factor}

In the measure theoretic context, if $W$ is a measure preserving transformation on a probability space $X$, we let $\I(W)$ denote the $\sigma$-algebra of $W$-invariant sets. 
For a system  $(X, \mathcal{X},\mu,S,T)$ with commuting transformations $S$ and $T$, let $X_{W}$ denote the factor associated to the $\sigma$-algebra $\I(W)$ and $\nu_{W}$ denote the projection of $\mu$ on $X_{W}$ for $W=S,T$ or $R$. For $(W_1,W_2) = (S,T)$, $(T,R)$ or $(S,R)$, let $\Z_{W_1,W_2}(X)$ denote the factor associated to the $\sigma$-algebra $\I(W_1)\vee \I(W_2)$. When there is no ambiguity, we write $\Z_{W_1,W_2}=\Z_{W_1,W_2}(X)$ for short. Let $\tilde{\pi}=\pi_{R}\times\pi_{T}\times\pi_{S}$ be the projection from $X^{3}$ onto $\Z_{S,T}\times\Z_{S,R}\times\Z_{T,R}$. 

The following lemma follows from Lemma 3.3 of \cite{DS}:
\begin{lem}\label{2p}
	Let $(X,\mathcal{X},\mu,S,T)$ be an ergodic system with commuting transformations $S$ and $T$. Then for $(W_1,W_2)=(S,T), (S,R)$ or $(T,R)$, we have
	$$ (\Z_{W_1,W_2},\I(W_1)\vee \I(W_2),\mu,W_{1},W_{2})\cong (X_{W_{1}}\times X_{W_{2}},\I(W_1)\times \I(W_2),\nu_{W_{1}}\times\nu_{W_{2}},\id\times W_{1},W_{2}\times \id),$$
	where $\id$ is the identity transformation.
\end{lem}

The following lemma was proved essentially in Section 3 of \cite{DS}.

\begin{lem} 
	Let $(X,\mathcal{X},\mu,S,T)$ be an ergodic system with commuting transformations $S$ and $T$. Then $(\Z_{S,T},\I(S)\vee \I(T),\mu,S,T)$ has a strictly ergodic topological model of the form $(Y\times Z,S\times \id, \id \times T)$, where $(Y,S)$ and $(Z,T)$ are strictly ergodic topological dynamical systems. 
\end{lem} 
\begin{rem}
	We refer to $(Y\times Z,S\times \id, \id \times T)$ as a \emph{product system}.
\end{rem}

\subsection{Host's magic systems and seminorms} \label{Subsec:Magic}

The following notions were introduced by Host in \cite{H}, inspired by the Austin's work \cite{Aus}, in order to study the $L^2$ convergence of multiple ergodic averages for commuting transformations. We briefly recall the construction for two commuting transformations $S$ and $T$. A more detailed exposition can be found in \cite{Chu,DS,H}. 

Let $\mu_S$ be the relative independent square of $\mu$ over $\mathcal{I}(S)$, {\em i.e.}  $$ \int_{X^2} f_0\otimes f_1 d\mu_S =\int_{X} \mathbb{E}(f_0\vert\mathcal{I}(S))\mathbb{E}(f_1\vert\mathcal{I}(S))d\mu$$
for all $f_0,f_1\in L^{\infty}(\mu)$. Then $\mu_S$ is a measure on $X^2$ invariant under $\id \times S$ and $g\times g$ for $g\in G=\langle S, T \rangle$. The measures $\mu_{T}$ and $\mu_{R}$ can be defined in a similar way.

Let $\mu_{S,T}$ denote the relative independent square of $\mu_{S}$ over $\mathcal{I}(T\times T)$, { \em i.e.}  
$$\int_{X^4}f_0\otimes f_1\otimes f_2\otimes f_3 d\mu_{S,T}=\int_{X^2} \mathbb{E}(f_0\otimes f_1\vert\mathcal{I}(T\times T))\mathbb{E}(f_2\otimes f_3\vert\mathcal{I}(T\times T))d\mu_{S}$$
for all $f_0,f_1,f_2,f_3\in L^{\infty}(\mu)$. Then $\mu_{S,T}$ is a measure on $X^4$ invariant under $\id\times S\times \id \times S$, $\id\times \id \times T \times T$ and under $g\times g\times g\times g $ for all $g\in G$. The measures $\mu_{S,R}$ and $\mu_{T,R}$ can be defined similarly. 

Write ${S^{\ast}}=\id\times S\times \id \times S$ and ${T^{\ast}}=\id\times \id \times T \times T$. Then $(X^4,\mathcal{X}^4,\mu_{S,T}, S^{\ast},T^{\ast})$ is a system with commuting transformations $S^{\ast}$ and $T^{\ast}$. The projection $\pi\colon(x_0,x_1,x_2,x_3)\to x_3$ defines a factor map between $(X^4,\mathcal{X}^4,\mu_{S,T},  S^{\ast},T^{\ast})$ and $(X,\mathcal{X},\mu,S,T)$. We remark that the system $(X^4,\mathcal{X}^{4},\mu_{S,T}, S^{\ast},T^{\ast})$ is not ergodic even when $(X,\mathcal{X},\mu,S,T)$ is. Nevertheless, it can be proved that $\mu_{S,T}$ is ergodic under the action spanned by $S^{\ast}$, $T^{\ast}$ and $g\times g\times g\times g$, $g\in G$. This can be deduced from page 12 in \cite{H} or can be derived as a consequence of Theorem 4.1 in \cite{DS}. Particularly (projecting into the first half), $\mu_W$ is ergodic under the action spanned by $\id \times W$ and $g\times g$, $g\in G$ for $W=S,T,R$.

\begin{defn}
	 For $f\in L^{\infty}(\mu)$, the \emph{Host seminorms} are  the quantities

$$\normm{f}_{\mu,W}=\Bigl(\int_{X^2}  f\otimes f d\mu_{W}\Bigr)^{1/2}$$
for $W=S,T,R$, and

$$\normm{f}_{\mu,W_1,W_2}=\Bigl(\int_{X^4}  f\otimes f\otimes f\otimes f d\mu_{W_1,W_2}\Bigr)^{1/4}$$
for $(W_1,W_2)=(S,T)$, $(S,R)$ or $(T,R)$.
\end{defn}

We summarize some results concerning these seminorms for later use.
\begin{thm}[\cite{H}, Sections 2,3,4; \cite{Chu} Section 3] \label{Cauchy}
Let $(W_1,W_2)=(S,T)$, $(S,R)$ or $(T,R)$. Then
\begin{enumerate}
\item ({\it Cauchy-Schwartz type inequality}) For $f_0,f_1,f_2,f_3\in L^{\infty}(\mu)$, we have
 $$\int_{X^4} f_0\otimes f_1 \otimes f_2\otimes f_3 d\mu_{W_1,W_2}\leq \normm{f_0}_{\mu,W_1,W_2}\normm{f_1}_{\mu,W_1,W_2}\normm{f_2}_{\mu,W_1,W_2}\normm{f_3}_{\mu,W_1,W_2}{\rm ;} $$

\item $\normm{\cdot}_{\mu,W_1,W_2}$ is a seminorm on $L^{\infty}(\mu)$. Moreover $\normm{\cdot}_{\mu,W_1,W_2}=\normm{\cdot}_{\mu,W_2,W_1}$ and $\normm{\cdot}_{\mu,W_1,W_2}\leq \|\cdot\|_{L^4(\mu)}{\rm ;}$

\item $\normm{f}_{\mu,W_1,W_2}= \lim\limits_{H\to \infty} \frac{1}{H} \sum\limits_{h=0}^{H-1} \normm{f\circ W_2^h \cdot f}_{\mu,W_1}= \lim\limits_{H\to \infty} \frac{1}{H} \sum\limits_{h=0}^{H-1} \Bigl\|\mathbb{E}(f\circ W_2^h \cdot f \vert\I(W_1))\Bigr\|_{L^2(\mu)}{\rm ;}$

\item \[\limsup_{N\to \infty} \left \|\frac{1}{N} \sum_{i=0}^{N-1} f_1(W_1^ix)f_2(W_2^ix) \right \|_{L^2(\mu)} \leq \min\{ \normm{f_1}_{\mu,W_1,W_1^{-1}W_2},\normm{f_2}_{\mu,W_2,W_1^{-1}W_2}\};\]
Particularly, \[\limsup_{N\to \infty} \left \|\frac{1}{N} \sum_{i=0}^{N-1} f_1(S^ix)f_2(T^ix) \right \|_{L^2(\mu)} \leq
\min\{\normm{f_1}_{\mu,S,R}, \normm{f_2}_{\mu,T,R}\}{\rm ;}\]

\item If $\pi\colon (X,\mathcal{X},\mu,W_1,W_2)\to (Y,\mathcal{Y},\nu,W_1,W_2)$ is a factor map, then 
 \[\normm{f}_{\nu,W_1,W_2}=\normm{f\circ \pi}_{\mu,W_1,W_2}{\rm ;}\]

\item If $\normm{f}_{\mu,W_1,W_2}=0$, then $\mathbb{E}(f\mid \I(W_1)\vee \I(W_2))=0$. 

\end{enumerate}
\end{thm}

\begin{defn}
 Let $(X,\mathcal{X},\mu,S,T)$ be a measure preserving system with commuting transformations $S$ and $T$. We say that $(X,\mathcal{X},\mu,S,T)$ is {\it magic} if
 $$\mathbb{E}(f\vert\mathcal{I}(S)\vee \mathcal{I}(T))=0 \text{ if and only if } \normm{f}_{\mu,S,T}=0.$$

\end{defn}
The connection between the Host measure $\mu_{S,T}$ and magic systems is:
\begin{thm}[\cite{H}, Theorem 2]\label{mig}
The system $(X^4,\mathcal{X}^4,\mu_{S,T}, {S^{\ast}},{T^{\ast}})$ defined in Section \ref{Subsec:Magic} is a magic extension system of $(X,\mathcal{X},\mu,S,T)$.
\end{thm}
 
The following theorem stated in \cite{DS} Section 3 strengthens this result.
\begin{thm} \label{MagicFree}
Let $(X,\mathcal{X},\mu,S,T)$ be an ergodic system with commuting transformations $S$ and $T$. Suppose that $S^i$ and $T^j$ are not the identity for any $i,j\in \mathbb{Z}\setminus \{0\}$ (equivalently, $(X,\mathcal{X},\mu,S)$ and $(X,\mathcal{X},\mu,T)$ are free). Then there exists a magic extension $(X',\mathcal{X}',\nu, {S^{\ast
}},{T^{\ast}})$ of $X$ such that the action of $\langle S^{\ast},T^{\ast} \rangle$ is free and ergodic on $X'$.
\end{thm}

\section{Building the topological model} \label{Sec:BuildTopModel}
In what follows we assume that $(X,\mathcal{X},\mu,S)$ and $(X,\mathcal{X},\mu,T)$ are free, since otherwise either $S$ or $T$ is periodic, and the averages we consider can be easily treated.
We study in detail the following topological structure.
\begin{defn}
	Let $(X,S,T)$ be a topological dynamical system with commuting transformations $S$ and $T$. We define $N_{S,T}(X)$ to be the set \[ N_{S,T}(X)=
	\overline{\{ (x,S^ix,T^ix)\colon  x\in X, i\in \mathbb{Z} \}} \subseteq X^3.\]
\end{defn}

Let $H_{S,T}\subseteq G^{3}$ be the group spanned by $\id\times S\times T$, $S\times S\times S$ and $T\times T\times T$. We remark that $H_{S,T}$ leaves invariant $N_{S,T}(X)$. Moreover, we have

\begin{prop}
	Let $(X,S,T)$ be a minimal topological dynamical system with commuting transformations $S$ and $T$. Then $(N_{S,T}(X),H_{S,T})$ is also a minimal topological dynamical system.
\end{prop} 
We omit the proof of this fact since it is similar to the one in page 46 of \cite{Glas}.

The main result  concerning this structure is the following. 

\begin{thm}\label{thm:uniquelyergodicmodel}
	Every ergodic system $({X},\mu,{S},{T})$ with commuting transformations ${S}$ and ${T}$ has an extension system $({X'},\mu',{S'},{T'})$ which admits a strictly ergodic model $(\wh{X'},S',T')$ such that $(N_{S',T'}(\wh{X'}), H_{S',T'})$ is also strictly ergodic.
\end{thm}

We prove this theorem in this section and show in Section \ref{Sec:PointwiseResults} how this result implies Theorems \ref{THM:1} and \ref{THM:2}.

\subsection{Models for Triple magic systems}

The following lemma shows that magic systems pass to the limit:

\begin{lem} \label{lem:InverseLimitMagic}
Let $(X,\mathcal{X},\mu,S,T)$ be the (measurable) inverse limits of the systems \\$(X_i,\mathcal{X}_{i},\mu_i,S_i,T_i), i\in\mathbb{N}$. If $(X_i,\mathcal{X}_{i},\mu_i,S_i,T_i)$ is magic for $S_i$ and $T_i$ for all $i\in \mathbb{N}$, then $(X,\mathcal{X},\mu,S,T)$ is magic for $S$ and $T$. 
\end{lem}

\begin{proof}
It suffices to prove that if $f$ is a function on $X$ with $\mathbb{E}(f\vert\I(S)\vee \I(T))=0$,  then $\normm{f}_{\mu,S,T}=0$ (the other implication is always true by Theorem \ref{Cauchy}-(6)). We regard $\mathcal{X}_i$  as the sub $\sigma$-algebra of $\mathcal{X}$ associated to the factor $(X_i,\mathcal{X}_i,\mu_i,S_i,T_i)$. Since $(X,\mathcal{X},\mu,S,T)$ is the inverse limit of $(X_i,\mathcal{X}_{i},\mu_i,S_i,T_i)$, we have that $\mathbb{E}(f\vert\mathcal{X}_i)$ converges in $L^1(\mu)$ to $f$ as $i\to\infty$. By Theorem \ref{Cauchy}-(2),
\[ \normm{f}_{\mu,S,T}=\lim_{i\to\infty} \normm{\mathbb{E}(f\vert\mathcal{X}_i)}_{\mu,S,T}. \]
Since $\mathbb{E}(f\vert\mathcal{X}_i)=\mathbb{E}(f\vert{X}_i)\circ \pi_i$, by Theorem \ref{Cauchy}-(5), it suffices to show that  $\normm{\mathbb{E}(f\vert{X}_i)}_{\mu,S_i,T_i}=0$ for every $i\in \mathbb{N}$. 
 
Since $X_i$ is magic for $S_i$ and $T_i$, it suffices to show that $\mathbb{E}(\mathbb{E}(f\vert{X}_i)\vert\mathcal{I}(S_i)\vee \mathcal{I}(T_i))=0$. By a density argument, it suffices to prove that $\int_{X_i} \mathbb{E}(f\vert{X}_i)(x)\cdot g(x)h(x)d\mu_{i}(x)=0$ for an $S_i$-invariant function $g$ and a $T_i$-invariant function $h$. By definition, we have that
\[ \int_{X_i}  \mathbb{E}(f\vert{X}_i)(x)\cdot g(x)h(x)d\mu_{i}(x)= \int_{X} f\cdot (h\circ \pi_i) \cdot (g\circ \pi_i) d\mu. \]
The latter integral is 0 since $\mathbb{E}(f\vert\I(S)\vee \I(T))=0$.
\end{proof}

\begin{defn}
	Let $(X,\mu,S,T)$ be an ergodic system with commuting transformations $S$ and $T$. We say that $(X,\mu,S,T)$ is \emph{triple magic} if $(X,\mu,S,T)$, $(X,\mu,S,R)$ and $(X,\mu,T,R)$ are magic systems. 
\end{defn}

The existence of triple magic extensions is guarenteed by the following property:

\begin{prop}\label{exis}
Every ergodic system $(X,\mu,S,T)$ with commuting transformations $S$ and $T$ admits a free, ergodic and triple magic extension.
\end{prop}

\begin{proof} 
	Let $(X,\mu,S,T)$ be an ergodic system with commuting transformations $S$ and $T$.
By Theorem \ref{MagicFree}, we can find  a free and ergodic extension  $(Y_1,\mu_{Y_1},S_{Y_1},T_{Y_1})$ which is magic for $S_{Y_1}$ and $T_{Y_1}$. Let $R_{Y_1}=S_{Y_1}^{-1}T_{Y_1}$. Then $(Y_1,\mu_{Y_1},S_{Y_1},T_{Y_1},R_{Y_1})$ is an extension of $(X,\mu,S,T,R)$. 

We can then find a free and ergodic extension $(W_1,\mu_{W_1},S_{W_1},R_{W_1})$ of $(Y_1,\mu_{Y_1},S_{Y_1},R_{Y_1})$ which is magic for $S_{W_1}$ and $R_{W_1}$. Let $T_{W_1}=S_{W_1}R_{W_1}$. Then $(W_1,\mu_{W_1},S_{W_1},T_{W_1},R_{W_1})$ is an extension of $(Y_1,\mu_{Y_1},S_{Y_1},T_{Y_1},R_{Y_1})$.

 Similarly, we can find a free and ergodic extension $(Z_1,\mu_{Z_1},T_{Z_1},R_{Z_1})$ of $(W_1,\mu_{W_1},T_{W_1},R_{W_1})$ which is magic for $T_{Z_1}$ and $R_{Z_1}$. Let $S_{Z_1}=T_{Z_1}R_{Z_1}$. Then $(Z_1,\mu_{Z_1},S_{Z_1},T_{Z_1},R_{Z_1})$ is an extension of $(W_1,\mu_{W_1},S_{W_1},T_{W_1},R_{W_1})$. We can then find a free ergodic extension $(Y_{2},\mu_{Y_2}, S_{Y_2},T_{Y_2})$ of $(Z_{1},\mu_{Z_1},S_{Z_1},T_{Z_1})$ which is magic for $S_{Y_2}$ and $T_{Y_2}$. 
 
 Repeating the process, we find a sequence of extensions $Y_i$, $W_i$ and $Z_i$ such that $Y_i$ is magic for $S_{Y_i}$ and $T_{Y_i}$, $W_i$ is magic for $S_{W_i}$ and $R_{W_i}$ and $Z_i$ is magic for $T_{Z_i}$ and $R_{Z_i}$. By Lemma \ref{lem:InverseLimitMagic}, their inverse limit $Y=\lim\limits_{\leftarrow} Y_i =\lim\limits_{\leftarrow} W_i=\lim\limits_{\leftarrow} Z_i$ is free, ergodic and magic for $S_{Y}$ and $T_{Y}$, for $S_{Y}$ and $R_{Y}$ and for $T_{Y}$ and $R_{Y}$.
\end{proof}

In the rest of this section, we assume $X$ is a free, ergodic and triple magic system obtained by Proposition \ref{exis}.
We review some properties of this system (see Chu \cite{Chu}, Section 4.2 for further details). For $W=S,T,R$, recall that $X_W$ is the factor associated to $\I(W)$. Let $\pi'_W\colon X\to X_{W}$, be the corresponding factor map.  Let $Y=X_S \times X_T \times X_R$ be endowed with the product $\sigma$-algebra and let $\pi\colon X\to X_S\times X_T\times X_R$ be the map given by $\pi(x)=(\pi'_{S}x,\pi'_{T}x,\pi'_{R}x)$. The transformations $S$, $T$ and $R$ are mapped to $S_Y=\id\times S\times T$, $T_Y=T\times \id \times T$ and $R_Y=T\times S^{-1}\times \id$, respectively.   
Let $\nu$ be the image of $\mu$ under the map $\pi$. Then the factor of $X$ associated to the $\sigma$-algebra $\I(S)\vee \I(T)\vee \I(R)$ is isomorphic to $(X_S\times X_T\times X_R, \nu)$. Let $Z$ be the factor spanned by the common eigenvalues of $S,T$ and $R$. Let $m$ be the image of $\nu$ on $Z\times Z\times Z$. 
Then $\nu$ is the conditionally independent product over $Z\times Z\times Z$. 

The proof of the following lemma is contained implicitly in Propositions 4.4 and 4.5 of \cite{Chu}:
\begin{lem} \label{lem:UniqueMeasureConcentrated}
	Let $\nu'$ be a measure on $X_S\times X_T\times X_R$ ergodic for the transformations $S_Y$ and $T_Y$. Let $m'$ be the image of $\nu'$ on $Z\times Z\times Z$. Then
	\[ \int_{X_S\times X_T\times X_R} f_1\otimes f_2\otimes f_3 d\nu'= \int_{Z\times Z\times Z} \mathbb{E}(f_1\vert Z)\otimes\mathbb{E}(f_2\vert Z)\otimes\mathbb{E}(f_3\vert Z) dm'. \] Moreover, there exists $c\in Z$ such that $m'$ is concentrated on the set \[ \{(z_1,z_2,z_3)\colon z_1+z_2-z_3=c\}\subseteq Z\times Z\times Z.\]   
\end{lem}

We are now ready to introduce the topological model needed for our question:

\begin{lem}\label{lem:magic3}
	There exist strictly ergodic models $\widehat{\Z_{S,T}}$, $\widehat{\Z_{S,R}}$, $\widehat{\Z_{T,R}}$ for $\Z_{S,T}$, $\Z_{S,R}$, $\Z_{T,R}$ and a strictly ergodic model $\widehat{X}$ for $X$ such that $\widehat{X}\to \widehat{\Z_{S,T}} $, $\widehat{X}\to \widehat{\Z_{S,R}} $ and $\widehat{X}\to \widehat{\Z_{T,R}}$ are topological models for $X\to \Z_{S,T}$, $X\to \Z_{S,R}$ and $X\to \Z_{T,R}$ respectively.
\end{lem}

\begin{rem}
	By Theorem \ref{Weiss}, we can always find topological models for the factor maps $X\to \Z_{S,T}$, $X\to \Z_{S,R}$ and $X\to \Z_{T,R}$, but we need that the topological model for $X$ in those three factors maps to be the same.
\end{rem} 

\begin{proof}[Proof of Lemma \ref{lem:magic3}]
	We remark that we can endow $Z$ with a natural topological structure (a compact abelian group). Let $\phi_S\colon X_S\to Z$, $\phi_T\colon X_T\to Z$ and $\phi_R\colon X_R\to Z$ be the factor maps. By Lemma \ref{lem:UniqueMeasureConcentrated} (and Proposition 4.5 in \cite{Chu}),  we may assume that
	\[ \int_{X_S\times X_T\times X_R} f_1\otimes f_2\otimes f_3 d\nu= \int_{Z\times Z\times Z} \mathbb{E}(f_1\vert Z)\otimes\mathbb{E}(f_2\vert Z)\otimes\mathbb{E}(f_3\vert Z) dm, \] 
	where $m$ is the Haar measure of the subgroup $H=\{(z_1,z_2,z_3):z_1+z_2-z_3=0\}\subseteq Z\times Z\times Z$. 
	By Theorem \ref{Weiss}, we can find strictly ergodic models $\widehat{\phi_S}\colon \widehat{X_S}\to Z$, $\widehat{\phi_T}\colon \widehat{X_T}\to Z$ and $\widehat{\phi_R}\colon \widehat{X_R}\to Z$ for the factor maps $\phi_S$, $\phi_T$ and $\phi_R$, respectively. 
	
	Let $\widehat{Y}$ be a minimal subsystem of 
	$$\Bigl\{ (x_1,x_2,x_3) \in \widehat{X}_S\times \widehat{X}_T \times \widehat{X}_R\colon \widehat{\phi_S}(x_1)+\widehat{\phi_T}(x_2)-\widehat{\phi_R}(x_3)=0\Bigr\}$$ for the transformations $S_{\widehat{Y}}=\id\times S \times T$ and $T_{\widehat{Y}}=T\times \id \times T$. By Lemma \ref{lem:UniqueMeasureConcentrated}, the projection of any ergodic measure on $\widehat{Y}$ is concentrated on $H$ and therefore is equal to $m$. So $(\widehat{Y},S_{\widehat{Y}},T_{\widehat{Y}})$ is a strictly ergodic model for $(Y,\nu,S_Y,T_Y)$. The projections into two different coordinates are topological models for the corresponding measurable projections. We get the announced result by taking a strictly ergodic model for the factor map $X\to Y$. 
\end{proof}

The following is the key property of this model (recall that $\tilde{\pi}=\pi_R\times \pi_T \times \pi_S$ is the projection from $X^3$ onto $\mathcal{Z}_{S,T}\times \mathcal{Z}_{S,R} \times \mathcal{Z}_{T,R}$). To ease notation we consider from the beginning that $X$ is its topological model given by Lemma \ref{lem:magic3} so all factors considered are topological and we omit writing ~$\widehat{\cdot}$ ~ everywhere.    

\begin{lem}\label{nz}
	Under the assumption of Lemma \ref{lem:magic3},
	$\Bigl(\tilde{\pi}\bigl(N_{S,T}({X})\bigr),\tilde{\pi}H_{S,T}\Bigr)$ is strictly ergodic.
	(Here $\tilde{\pi}H_{S,T}$ is the projection of $H_{S,T}$ onto $\tilde{\pi}\bigl(N_{S,T}({X})\bigr)$)  
	%(2)  For every $x\in \Z_{S,T}$, $\ol{\O((x,x,x),\s)}$ has a unique $\s$-invariant measure.
\end{lem}
\begin{proof}
	By Lemma \ref{2p} and \ref{lem:magic3}, the factors 
	$({\Z_{S,T}}\cong{X}_{S}\times {X}_{T}, \id\times S, T \times \id)$, $({\Z_{S,R}}\cong {X}_{S}\times {X}_{R}, \id\times T, T \times \id)$ and $({\Z_{T,R}}\cong {X}_{T}\times {X}_{R}, S\times \id, \id \times S)$  are strictly ergodic systems. Here we slightly abuse notation and write with the same letters the projections of $S$ and $T$ onto the factors $X_S$, $X_T$, $X_R$ etc.  
%	$(\Z_{S,T},S,T)$, $(\Z_{S,R},S,T)$ and $(\Z_{T,R},S,T)$, respectively. 
	We have the isomorphism
	\begin{equation}\nonumber
	\begin{split}
	& {\Z_{S,T}}\times {\Z_{S,R}}\times {\Z_{T,R}}\cong {X}_{S}\times {X}_{T}\times {X}_{S}\times {X}_{R}\times {X}_{T}\times {X}_{R},
	\\&                                             \id\times S\times T\leftrightarrow \id\times \id\times      \id\times        S\times       \id\times S,
	\\&                                             S\times S\times S\leftrightarrow \id\times S\times       \id\times        S\times       S\times S,
	\\&                                             T\times T\times T\leftrightarrow  T\times \id\times      T\times         S\times       \id\times S.
	\end{split}
	\end{equation} 
	
	Since $ N_{S,T}({X})$ is the orbit closure of diagonal points, it is easy to see that $\tilde{\pi}\bigl(N_{S,T}({X})\bigr)$ is a subsystem of ${\Z_{S,T}}\times {\Z_{S,R}}\times {\Z_{T,R}}$ whose 1,2,4-th coordinates are the same as the 3,5,6-th coordinates, respectively. 
	So $\tilde{\pi}(N_{S,T}({X}))$ is isomorphic to a subsystem of
	${X}_{S}\times {X}_{T}\times {X}_{R}$. The group $H_{S,T}$ is generated by  
	$\id\times S\times T$, $S\times S\times S$ and $T\times T\times T$
	and their projection onto $\tilde{\pi}\bigl(N_{S,T}({X})\bigr)$ is then generated by $\id\times \id\times S$, $\id\times T\times S$ and $T\times \id\times S$. 

But the group generated by $\id\times \id\times S$, $\id\times T\times S$ and $T\times \id\times S$ is the same as the one generated by $\id\times \id\times S$, $\id\times T\times \id$ and $\id \times \id\times S$, so the system $\Bigl(\tilde{\pi}\bigl(N_{S,T}({X})\bigr),\tilde{\pi}H_{S,T}\Bigr)$ is isomorphic to a subsystem of $({X}_{S}\times {X}_{T}\times {X}_{R}, \id\times \id\times S, \id\times T\times \id,\id \times \id\times S)$. But this latter system is a product of three strictly ergodic systems and thus it is strictly ergodic as well (see for instance \cite{DS}, Section 4).      	
We conclude that $\Bigl(\tilde{\pi}\bigl(N_{S,T}({X})\bigr),\tilde{\pi}H_{S,T}\Bigr)$ is actually isomorphic to $({X}_{S}\times {X}_{T}\times {X}_{R}, \id\times \id\times S, \id\times T\times \id,\id \times \id\times S)$ and we are done.
\end{proof}

\begin{rem} \label{Rem:CoordinatesDetermine}
It is worth noting that in the projections into $\Z_{S,T}$ and $\Z_{T,R}$ determine the projection into $X_S$, $X_T$ and $X_R$. Consequently, they determine the projection into $\Z_{S,T}$. 
\end{rem}

\subsection{Strictly ergodic model for $N_{S,T}(X)$}

By Lemma \ref{nz}, (if $X$ is its model in Lemma \ref{lem:magic3}) there is a unique invariant measure  $\xi$ on $\bigl(\tilde{\pi}(N_{S,T}(X)),\tilde{\pi}H_{S,T}\bigr)$. The projection of $\xi$ into the first coordinate is the unique invariant measure $\nu_{S,T}$ on ${\Z_{S,T}}$, so we may consider the disintegration of $\xi$ over $\nu_{S,T}$. 
\begin{equation} \label{Eq:xi}
\xi=\int_{\Z_{S,T}}\d_{s}\times \eta_{s} d\nu_{S,T}(s)
\end{equation}
%be the disintegration of $\xi$ over $\nu_{S,T}$ ({\em i.e. } the disintegration given by the projection into the first coordinate of $\tilde{\pi}(N_{S,T}(X)))$.

To study further this disintegration we need some lemmas. 

\begin{lem}  \label{Lem:boundMeasurable} Let $f_0,f_1\in L^{\infty}(\mu)$ with $\Vert{f_0}\Vert_{\infty}\leq 1$ and $\Vert{f_1\Vert}_{\infty}\leq 1$. Then
\[\Vert \mathbb{E}(f_0 \otimes f_1\vert \mathcal{I}(S\times T) \Vert_{L^2(\mu_R)} \leq \min\{ |||f_0|||_{\mu,R,S},|||f_1|||_{\mu,R,T} \}.  \]
\end{lem}

\begin{proof}
By the Von Neumann Ergodic Theorem, we have that
\[\left \Vert \mathbb{E}(f_0 \otimes f_1\vert \mathcal{I}(S\times T) \right \Vert_{L^2(\mu_R)}=\lim_{N\to\infty} \left \Vert \frac{1}{N} \sum_{i=0}^{N}  f_0 \otimes f_1(S^i \times T^i) \right \Vert_{L^2(\mu_R)}.   \]
Applying van der Corput Lemma (see \cite{HK05} Appendix D for example), this limit average is bounded by 

\[\limsup_{H\to \infty} \frac{1}{H}\sum_{h=0}^{H-1} \Bigl\vert\limsup_{N\to \infty} \frac{1}{N} \sum_{i=0}^{N-1} \int_{X^2} f_0 \otimes f_1(S^{h+i} \times T^{h+i})\cdot f_0 \otimes f_1(S^i \times T^i)d\mu_R\Bigr\vert.   \]

Using the invariance of $\mu_R$ under $S\times T$ this expression equals 

\[\limsup_{H\to \infty}\frac{1}{H} \sum_{h=0}^{H-1} \Bigl\vert\limsup_{N\to \infty} \frac{1}{N} \sum_{i=0}^{N-1} \int_{X^2} f_0\cdot f_0\circ S^{h} \otimes  f_1 \cdot f_1\circ T^h d\mu_R \Bigr\vert.  \]

On the other hand, \[\int_{X^2} f_0\cdot f_0\circ S^{h} \otimes  f_1 \cdot f_1\circ T^h d\mu_R=\int_{X} \mathbb{E}(f_0\cdot f_0\circ S^h\vert \mathcal{I}(R))\mathbb{E}(f_1\cdot f_1\circ T^h \vert \mathcal{I}(R)) d\mu.\]

Using Cauchy Schwartz in this last expression, we get the bounds  
\[\limsup_{H\to \infty} \sum_{h=0}^{H-1}\frac{1}{H} \left \|\mathbb{E}(f_0\cdot f_0\circ S^{h}\vert \mathcal{I}(R)) \right \|_{L^2(\mu)} ~ \text{ and } ~ \limsup_{H\to \infty} \sum_{h=0}^{H-1}\frac{1}{H} \left \|\mathbb{E}(f_1\cdot f_1\circ T^{h}\vert \mathcal{I}(R)) \right \|_{L^2(\mu)}. \]
By Theorem \ref{Cauchy}-(3), these quantities converge to $|||f_0|||_{\mu,R,S}$ and $|||f_1|||_{\mu,R,T}$ and we are done.
\end{proof}

This lemma immediately implies the following:
\begin{lem}  Let $(X,\mathcal{X},\mu,S,T)$ be an ergodic triple magic system with commuting transformations $S$ and $T$. Let $f_0,f_1\in L^{\infty}(\mu)$. Then
	$$ \mathbb{E}(f_0 \otimes f_1 \vert \mathcal{I}(S\times T)) =\mathbb{E}\Bigl(\mathbb{E}(f_0\vert \Z_{S,R}) \otimes \mathbb{E}(f_1\vert \Z_{T,R})\big\vert \mathcal{I}(S\times T) \Bigr). $$ Consequently, $$(X^{2},\mathcal{I}(S\times T),\mu_{R})\cong\bigl(\Z_{S,R}\times\Z_{T,R},\mathcal{I}(S\times T),(\pi_{T}\times \pi_{S})_{*}(\mu_{R})\bigr).$$	
\end{lem}

\begin{proof}
It suffices to show that $\mathbb{E}(f_0\otimes f_1 \mid \I(S\times T))=0$ whenever $\mathbb{E}(f_0\mid \Z_{S,R})=0$ or $\mathbb{E}(f_1 \mid \Z_{T,R})=0$. Lemma \ref{Lem:boundMeasurable} gives us exactly this result.
\end{proof}

The next lemma is one of the key ingredients of the proof:

\begin{lem}\label{ist}  Let $(X,\mathcal{X},\mu,S,T)$ be an ergodic triple magic system with commuting transformations $S$ and $T$. Then
	$$\bigl(\Z_{S,R}\times\Z_{T,R},\I(S\times T),(\pi_{T}\times \pi_{S})_{\ast}\mu_R\bigr)\cong(\Z_{S,T},\W,\nu_{S,T}).$$
\end{lem}

\begin{proof}
	We first show that $\mathbb{E}(f_0\otimes f_1\mid \I(S\times T))$ is measurable with respect to $\I(S)\times \I(T)$ when $f_0$ is measurable with respect to $\Z_{S,R}$ and $f_1$ is measurable with respect to $\Z_{T,R}$. By a density argument, it suffices to prove it for the case when $f_0=h_0g_0$,  $f_1=h_1g_1$, where $h_0$ is $S$-invariant, $h_1$ is $T$-invariant and $g_0,g_1$ are $R$-invariant.
By the Birkhoff Ergodic Theorem, we have that 
\[\mathbb{E}(f_0\otimes f_1\mid \I(S \times T))=h_0 \otimes h_1\cdot \mathbb{E}(g_0\otimes g_1\mid \I(S\times T)).\] 

Since $g_0$ and $g_1$ are $R$-invariants,  the function $\mathbb{E}(g_0\otimes g_1\mid \I(S\times T))$ is invariant under $\id \times R$, $S\times S$ and $T\times T$. Since the measure $\mu_R$ is ergodic under these transformations (see Section \ref{Subsec:Magic}), $\mathbb{E}(g_0\otimes g_1\mid \I(S\times T))=\int g_0\otimes g_1 d\mu_R=\int g_0g_1d\mu$ is a constant. Thus  \[\mathbb{E}(f_0\otimes f_1\mid \I(S \times T))=\left (\int g_0g_1d\mu \right)h_0\otimes h_1, \]
which is clearly $\I(S)\times \I(T)$ measurable. We remark the measure $\mu_R$ on $\I(S)\times \I(T)$ coincides with the product measure $\mu\otimes \mu$. By Lemma \ref{2p}, this system is isomorphic to $(\Z_{S,T},\W,\nu_{S,T})$ and we are done. 
\end{proof}

In conclusion,  we have

\begin{cor}\label{equiv} 
 Let $(X,\mathcal{X},\mu,S,T)$ be an ergodic triple magic system with commuting transformations $S$ and $T$.  The following probability spaces are isomorphic:
  \begin{itemize}
  \item $(X^{2},\mathcal{I}(S\times T),\mu_{R})$;
  \item $\bigl(\Z_{S,R}\times\Z_{T,R},\mathcal{I}(S\times T),(\pi_{T}\times \pi_{S})_{*}(\mu_{R})\bigr)$;
  \item $(\Z_{S,T},\W,\nu_{S,T})$.
  \end{itemize}
\end{cor}

Let $(X,S,T)$ be a topological dynamical system with commuting transformations $S$ and $T$. 
For $W=S,T$ or $R$, denote $$\Q_{W}(X)=\overline{\{(x,W^{i}x)\colon x\in X, i\in\mathbb{Z}\}}\subseteq X^2$$ and let 
$G_{W}$ be the subgroup of $G\times G$ spanned by $g\times g$, $g\in G$ and $\id\times W$.
The following result from \cite{DS} replaces section 4.1.2 in \cite{HSY}:

\begin{thm}[\cite{DS} Theorem 4.1]  \label{TH:Q_SU.E.}
Let $(X,\mathcal{X},\mu,S,T)$ be a magic system (for $S$ and $T$). If the projection $X\to \Z_{S,T}$ is continuous (we assume the spaces are topological), then $\Q_S(X)$ and $\Q_{T}(X)$ are uniquely ergodic with measures $\mu_S$ and $\mu_T$ respectively.  
\end{thm}

We are now ready to prove Theorem \ref{thm:uniquelyergodicmodel}:

\begin{proof}[Proof of Theorem \ref{thm:uniquelyergodicmodel}]
 By passing to an extension, we may assume that $(X,\mathcal{X},\mu,S,T)$ is free, ergodic and triple magic by Proposition \ref{exis}. We may assume that $({X},S,T)$ is its model given by Lemma \ref{lem:magic3} and  and then by Lemma \ref{nz} $\Bigl(\tilde{\pi}\bigl(N_{S,T}({X})\bigr),\tilde{\pi}H_{S,T}\Bigr)$ is strictly ergodic. All the factors considered are topological so for convenience we do not write the symbol~  $\widehat{\cdot}$.

 Suppose that $\l$ is an $H_{S,T}$-invariant measure on $N_{S,T}(X)$. Let $$p_{1}\colon (N_{S,T}(X),H_{S,T})\to (X,G)$$ be the projection onto the first coordinate and $$p_{2}\colon (N_{S,T}(X),H_{S,T})\to (\Q_{R}(X),G_{R})$$ be the projection onto the last two coordinates. 
  By the unique ergodicity of $(X,G)$ and Theorem \ref{TH:Q_SU.E.}, $(p_{1})_{*}(\l)=\mu$ and $(p_{2})_{*}(\l)=\mu_{R}$. So we may assume that
$$\l=\int_{\Q_{R}(X)}\l_{\bold{x}}\times \d_{\bold{x}}d\mu_{R}(\bold{x})$$
is the disintegration of $\l$ over $\mu_{R}$. We remark that the measure $\l_{\bold{x}}$ has a support included in $\{c: (c,\bold{x})\in N_{S,T}(X)\} \subseteq X$. Since $\l$ is $(\id\times S\times T)$-invariant, we have that
$$\l=(\id\times S\times T)_{*}\l=\int_{\Q_{R}(X)}\l_{\bold{x}}\times \d_{(S\times T)\bold{x}}d\mu_{R}(\bold{x})=\int_{\Q_{R}(X)}\l_{(S\times T)^{-1}\bold{x}}\times \d_{\bold{x}}d\mu_{R}(\bold{x}).$$
So 
$$\l_{(S\times T)\bold{x}}=\l_{\bold{x}}$$
for $\mu_{R}-a.e.$ $\bold{x}\in\Q_{R}(X)$.
Define
$$F\colon (\Q_{R}(X),\mu_{R},S\times T)\to M(X)$$
by $F(\bold{x})=\l_{\bold{x}}$. Then $F$ is $\I(S\times T)$-measurable. By Corollary \ref{equiv}, we can write $\l_{\bold{x}}=\l_{\phi(\bold{x})}$ for $\mu_{R}$-a.e. $\bold{x}\in\Q_{R}(X)$, where $\phi\colon (X^{2},\mathcal{X}^{2},\mu_{R})\to (\Z_{S,T},\W,\nu_{S,T})$ is the factor map.

Let
$$\mu_{R}=\int_{\Z_{S,T}}m_{x}d\nu_{S,T}(x)$$
be the disintegration of $\mu_{R}$ over $\nu_{S,T}$. Then
  \begin{equation}\nonumber
    \begin{split}
      &\l=\int_{\Q_{R}(X)}\l_{\bold{x}}\times \d_{\bold{x}}d\mu_{R}(\bold{x})=\int_{\Q_{R}(X)}\l_{\phi(\bold{x})}\times \d_{\bold{x}}d\mu_{R}(\bold{x})
      \\&=\int_{\Z_{S,T}}\int_{\Q_{R}(X)}\l_{s}\times \d_{\bold{x}}dm_{s}(\bold{x})d\nu_{S,T}(s)=\int_{\Z_{S,T}}\l_{s}\times m_{s}d\nu_{S,T}(s).
    \end{split}
  \end{equation}
  So $$\tilde{\pi}_{*}(\l)=\int_{\Z_{S,T}}(\pi_{R})_{*}\l_{s}\times (\pi_{T}\times\pi_{S})_{*}m_{s}d\nu_{S,T}(s).$$
  On the other hand, by \eqref{Eq:xi}, we have 
  $$\tilde{\pi}_{*}(\l)=\xi=\int_{\Z_{S,T}}\d_{s}\times \eta_{s}d\nu_{S,T}(s).$$
The measure $(\pi_{R})_{*}\l_{s}$ has a support included in $\{\pi_{R}(c):(c,\bold{x})\in N_{S,T}(X), s=\phi(\bold{x}) \}$ and since $\phi({\bold{x}})$ determines $\pi_{R}(c)$ (see Remark \ref{Rem:CoordinatesDetermine}), we have that $(\pi_{R})_{*}\l_{s}=\d_{s}$ for $\nu_{S,T}-a.e.$ $s\in\Z_{S,T}$. Since $(p_{1})_{*}(\l)=\mu$, we have that
  $$\mu=\int_{\Z_{S,T}}\l_{s}d\nu_{S,T}(s).$$
  Let $$\mu=\int_{\Z_{S,T}}\theta_{s}d\nu_{S,T}(s)$$ be the disintegration of $\mu$ over $\nu_{S,T}$. Since $\tilde{\pi}_{*}\l_{s}=\tilde{\pi}_{*}\theta_{s}=\d_{s}$ for $\nu_{S,T}-a.e.$  $s\in \Z_{S,T}$, by the uniqueness of disintegration, we have that $\l_{s}=\theta_{s}$ for $\nu_{S,T}-a.e.$  $s\in \Z_{S,T}$. Therefore
  $$\l=\int_{\Z_{S,T}}\l_{s}\times m_{s}d\nu_{S,T}(s)=\int_{\Z_{S,T}}\theta_{s}\times m_{s}d\nu_{S,T}(s),$$
  which is a uniquely determined measure since $\Z_{S,T}$ is uniquely ergodic.
\end{proof}

\section{Pointwise results} \label{Sec:PointwiseResults}
As in the previous section, we assume that $(X,\mathcal{X},\mu,S)$ and $(X,\mathcal{X},\mu,T)$ are free.   
In this section, whenever $(X,\mathcal{X},\mu,S,T)$ is a triple magic system, we assume it is its strictly ergodic model given by Theorem \ref{thm:uniquelyergodicmodel}, and use $\lambda_{S,T}$ to denote the unique ergodic measure of $(N_{S,T}({X}),H_{S,T})$.

\subsection{Proof of Theorem \ref{THM:1}}
We are now ready to prove Theorem \ref{THM:1}. In fact, if the system is triple magic, we can obtain an explicit limit:

\begin{thm} \label{Prop:PointConverBigAverage}
	 
Let 
$(X,\mathcal{X},\mu,S,T)$ be an ergodic system with commuting transformations $S$ and $T$.
 Then for all $f_0,f_1,f_3\in L^{\infty}(\mu)$, the average

\[\frac{1}{N^3} \sum_{i,j,k=0}^{N-1} f_0(S^jT^kx)f_1(S^{i+j}T^k x)f_2(S^j T^{i+k}x) \]
converges for $\mu$-a.e. $x\in X$ as $N\to \infty$. Moreover, if the system is free, ergodic and triple magic, then the limit is $\int f_0 \otimes f_1\otimes f_2 d\lambda_{S,T}$.
\end{thm}

\begin{proof}
By Theorem \ref{thm:uniquelyergodicmodel}, $X$ has an extension $X'$ which has a topological model $\wh{X}'$ such that $(N_{S',T'}(\wh{X}'),H_{S',T'})$ is strictly ergodic.
It suffices to work on $\wh{X}'$ instead of $X$. So for convenience we assume $X=\wh{X}'$ in the proof.

 Fix $\epsilon>0$.
Let $\widehat{f}_0$, $\widehat{f}_1$ and $\widehat{f}_2$ be continuous functions on $X$ such that $\| f_i-\widehat{f}_i \|_{L^1(\mu)} \leq \epsilon$. We assume  without loss of generality that the $L^{\infty}(\mu)$ norms of $f_{i},\wh{f}_{i},i=0,1,2$ are bounded by 1. For any functions $h_0,h_1,h_2$, write $$\mathbb{E}_N(h_0,h_1,h_2)(x)=\frac{1}{N^3} \sum_{i,j,k=0}^{N-1} h_0(S^jT^kx)h_1(S^{i+j}T^k x)h_2(S^j T^{i+k}x)$$ and $$I(h_0,h_1,h_2)=\int h_0\otimes h_1 \otimes h_2 d\lambda_{S,T}.$$

By telescoping, we have that 
\begin{align*}
& \quad\bigl\vert\mathbb{E}_N(f_0,f_1,f_2)(x)-\mathbb{E}_N(\widehat{f}_0,\widehat{f}_1,\widehat{f}_2)(x)\bigr\vert \\
& \leq \frac{1}{N^3}\sum_{i,j,k=0}^{N-1}\bigl\vert f_0(S^jT^kx)-\widehat{f}_0(S^jT^kx)\bigr\vert + \frac{1}{N^3}\sum_{i,j,k=0}^{N-1}\bigl\vert f_1(S^{j+i}T^jx)-\widehat{f}_1(S^{j+i}T^kx)\bigr\vert \\ 
&\quad + \frac{1}{N^3}\sum_{i,j,k=0}^{N-1}\bigl\vert f_2(S^jT^{i+k}x)-\widehat{f}_2(S^jT^{i+k}x)\bigr\vert
\end{align*}

By the Pointwise Ergodic Theorem (for abelian actions), the three terms on the right hand side converge almost everywhere to $\| f_0-\widehat{f}_0 \|_{L^1(\mu)}$, $\| f_1-\widehat{f}_1 \|_{L^1(\mu)}$ and $\| f_2-\widehat{f}_2 \|_{L^1(\mu)}$, respectively.

Again by telescoping, we deduce that 
\[ \bigl\vert I(f_0,f_1,f_2)-I(\widehat{f}_0,\widehat{f}_1,\widehat{f}_2)\bigr\vert\leq \| f_0-\widehat{f}_0 \|_{L^1(\mu)}+\| f_1-\widehat{f}_1 \|_{L^1(\mu)} + \| f_2-\widehat{f}_2 \|_{L^1(\mu)}. \]
On the other hand, since $(N_{S,T}(X),H_{S,T})$ is uniquely ergodic, we have that 
 \[\lim_{N\to\infty}\mathbb{E}_N(\widehat{f}_0,\widehat{f}_1,\widehat{f}_2)(x) =I(\widehat{f}_0,\widehat{f}_1,\widehat{f}_2) \text{ for every } x\in X.\] Thus for $\mu$-a.e. $x\in X$, we have \[ \limsup_{N\to\infty}  \bigl\vert\mathbb{E}_N(f_0,f_1,f_2)-I(f_0,f_1,f_2)\bigr\vert \leq 6 \epsilon.\] The result follows since $\epsilon$ is arbitrary.
\end{proof}

\subsection{Measurable distal systems} \label{Sec:MultipleAverage} 
In this section we study the properties of distal systems. 
We start with some definitions (see \cite{Glas} Chapter 10 for further details):  
\begin{defn}
Let $\pi\colon (X,\mathcal{X},\mu,G)\to (Y,\mathcal{Y},\nu,G)$ be a factor map between two ergodic systems. We say $\pi$ is an \emph{isometric} extension if there exist a compact group $H$, a closed subgroup $\Gamma$ of $H$, and a cocycle $\rho\colon G \times Y\to H$ such that $(X,\mathcal{X},\mu,G)\cong (Y\times H/\Gamma,\mathcal{Y}\times\mathcal{H}, \nu\times m, G)$, where $m$ is the Haar measure on $H/\Gamma$, $\mathcal{H}$ is the Borel $\sigma$-algebra on $H/\Gamma$, and that for all $g \in G$, we have
$$g(y,a\Gamma)=(g y, \rho(g,y)a\Gamma).$$

In this case, we say that $\pi\colon (X,\mathcal{X},\mu,G)\to (Y,\mathcal{Y},\nu,G)$ is an {\it isometric extension} with fiber $H/\Gamma$ and cocycle $\rho$. We denote  $X$ by $Y\times_{\rho} H/\Gamma$.
\end{defn}

\begin{rem} \label{Rem:TopologyH} Let ${\rm Aut}(X,\mu)$ be the group of measurable transformations of $X$ which preserve the measure $\mu$, endowed with the weak topology of convergence in measure, meaning that $h_n\to h\in {\rm Aut}(X,\mu) $ if and only if $\| f\circ h-f\circ h_n \|_{L^2(\mu)}\to 0$ for all $f\in L^2(\mu)$. Under this topology, ${\rm Aut}(X,\mu)$ is a Polish group (see \cite{BK}, Chapter 1). An important fact of isometric extensions is  that the group $H$ can be regarded as a compact subgroup of ${\rm Aut}(X,\mu)$, considering its inclusion on ${\rm Aut}(X,\mu)$  and this is  independent of the choice of models for $X$. This follows basically from the fact that measurable morphisms between Polish groups are automatically continuous (see \cite{BK}, Chapter 1, Theorem 1.2.6). 
\end{rem}

\begin{rem}
For every isometric extension $\pi\colon X\to Y$ with fiber $H/\Gamma$ and  measurable function  $f$ on $(X,\mu)$, the conditional expectation of $f$ (as a function on $(X,\mu)$) with respect to $Y$ is 
\[\mathbb{E}(f\vert \mathcal{Y})(x)=\int_{H} f(hx)dm(h).\]

Equivalently (as a function on $(Y,\mathcal{Y},\nu)$),

\[\mathbb{E}(f\vert Y)(y)=\int_{H} f(hx)dm(h) \quad \text{ for all }\pi(x)=y. \]   
  
\end{rem}

\begin{defn}
Let $\pi\colon (X,\mathcal{X},\mu,G)\to (Y,\mathcal{Y},\nu,G)$ be a factor map between two ergodic systems. We say $\pi$ is a \emph{distal extension} if there exist a countable ordinal $\eta$ and a directed family of factors $(X_{\theta},\mu_{\theta}, G), \theta\leq\eta$ such that
\begin{itemize}
\item  $X_{0}=Y$, $X_{\eta}=X$;
\item  For $\theta<\eta$, the extension $\pi_{\theta}\colon X_{\theta+1}\to X_{\theta}$ is isometric and is not an isomorphism;
\item  For a limit ordinal $\l\leq \eta$, $X_{\l}=\lim\limits_{\leftarrow \theta<\l} X_{\theta}$.
\end{itemize}
We say $X$ is a \emph{distal system} if $X$ is a distal extension of the trivial system.
\end{defn}

An alternative definition of a measurable distal system is formulated using separating sieves:
\begin{defn}
Let $\pi\colon (X,\mathcal{X},\mu,G)\to (Y,\mathcal{Y},\nu,G)$ be a factor map between two ergodic systems. A {\it separating sieve} for $X$ over $Y$ is a sequence of measurable subset $\{A_i\}_{i\in \mathbb{N}}$ with $A_{i+1}\subseteq A_{i}$, $\mu(A_i)>0$ and $\mu(A_i)\to 0$ such that there exists a measurable subset $X'\subseteq X$, $\mu(X')=1$ with the following property: for $x,x'\in X'$, if $\pi(x)=\pi(x')$ and for every $i\in \mathbb{N}$ there exists $g\in G$ such that $gx,gx'\in A_i$, then $x=x'$.
\end{defn}

\begin{prop}(\cite{Glas}, Chapter 10)
	Let $(X,\mathcal{X},\mu,G)$ be an extension of $(Y,\mathcal{Y},\nu,G)$. Then $X$ is a distal extension of $Y$ if and only if there exists a separating sieve for $X$ over $Y$.
\end{prop}

The following proposition extends Proposition \ref{exis}:

\begin{prop} \label{Prop:DistalMagicExtension}
Every ergodic distal system $(X,\mathcal{X},\mu,S,T)$ with commuting transformations $S$ and $T$ admits a free, ergodic, triple magic extension (as in Section \ref{Sec:BuildTopModel}) which is also distal. 
\end{prop}

To prove this results we need the following proposition, which we think is of independent interest.   We state it here in complete generality. 

\begin{prop} \label{prop:DistalDecomposition}
Let $(X,\mu,G)$ be an ergodic distal measure preserving system where $G$ is an abelian group action. Let $H$ be an infinite subgroup of $G$ and let $\mu=\int \mu_{x} d\mu(x)$ be the ergodic decomposition of $\mu$ under the action of $H$ ({\em i.e.} the disintegration of $\mu$ over the $\sigma$-algebra of $H$-invariant sets). Then for $\mu$-a.e $x \in X$, the measure $\mu_{x}$ is ergodic and distal for the action of $H$.
\end{prop}

\begin{proof} Since $(X,\mu,G)$ is ergodic and distal, there exists a separating sieve $\{A_i\}_{i\in \mathbb{N}}$ for $X$. By ergodicity, $\mu(\bigcup_{g\in G} g A_i)=1$ for all $i\in \mathbb{N}$. So for $\mu$-a.e. $x \in X$,  $\mu_{x} (\bigcup_{g\in G} g A_i)=1$ for all $i\in \mathbb{N}$. Since we are disintegrating over the $\sigma$-algebra of $H$-invariant sets, we have that $\mu_{x}$ is ergodic under the action of $H$ for $\mu$-a.e $x\in X$. It suffices to show that $(X,\mu_{x},H)$ is distal. We may assume that $\mu_{x}$ is a non-atomic measure, since otherwise it is a rotation on a finite set. 

We claim that we can construct a separating sieve $\{A_i^{x}\}_{i\in \mathbb{N}}$ for $\mu$-a.e. $x\in X$. To do so, for $\mu$-a.e $x\in X$, we can find $g_1 \in G$ such that $\mu_{x}(g_1A_1)>0$. Set $A_1^{x}=g_1A_1$.  Since $\mu_{x}$ is non-atomic, we can find $B_1^{x}\subseteq A_1^{x}$ with the half of the measure of $A_1^{x}$ and find $g_2 \in G$ such that $\mu_{x}(B_1^{x}\cap g_2 A_2)>0$. We set $A_2^{x}=B_1^{x}\cap g_2A_2$. Inductively, if we have defined $A_i^{x}$, we take a subset $B_{i}^{x}$ with the half of its measure, then pick $g_{i+1}\in G$ such that $\mu_{x}(B_i^{x} \cap g_{i+1}A_{i+1})>0$ and set $A_{i+1}^{x}=B_i^{x}\cap g_{i+1}A_{i+1}$. 
By construction we have that $A_{i+1}^{x}\subseteq A_i^{x}$, $\mu_{x}(A_{i}^{x})>0$ and $\mu_{x}(A_{i}^{x})\to 0$ as $i\to\infty$. It is now easy to check that $\{A_i^{x}\}_{i\in \mathbb{N}}$ is a separating sieve for $(X,\mu_{x},H)$.
\end{proof}

\begin{rem}
The statement of Proposition \ref{prop:DistalDecomposition} is trivial in the topological setting (subactions of topological distal are topological distal systems), but we did not find a reference in the measurable case.
\end{rem}

\begin{proof}[Proof of Proposition \ref{Prop:DistalMagicExtension}]

Let $\{A_i\}_{i\in \mathbb{N}}$ be a separating sieve for $(X,\mu,G)$ (over the trivial system).

{\bf Claim: } $\{A_i\times A_i\times A_i\times A_i\}_{i\in \mathbb{N}}$ is a separating sieve for $(X^{4}, \mathcal{X}^4,\mu_{S,T})$ for the action spanned by $S^{\ast}=\id\times S\times\id \times S$, $T^{\ast}=\id\times \id \times T\times T$ and the diagonals $S\times S\times S\times S$, $T\times T\times T\times T$. For convenience let $\mathcal{G}_{S,T}$ denote this group.
Note that the Jensen inequality implies that 

\[\mu_{S}(A_i\times A_i)^{1/2}=\Bigl(\int_{X} |\mathbb{E}(1_{A_i}\vert I(S))|^2 d\mu \Bigr)^{1/2} \geq \int_{X} \mathbb{E}(1_{A_i}\vert I(S))d\mu=\mu(A_i)>0.\] 

Similarly 
\[\mu_{S,T}(A_i\times A_i\times A_i\times A_i)^{1/4}\geq \mu_{S}(A_i\times A_i)^{1/2}\geq \mu(A_i)>0.\]

So \[0<\mu_{S,T}(A_i\times A_i\times A_i\times A_i)\leq \mu_{S,T}(A_i\times X\times X\times X)=\mu(A_i)\to 0.\] 

On the other hand, let $(x_0,x_1,x_2,x_3), (y_0,y_1,y_2,y_3)\in X^{4}$ so that for all $i\in\mathbb{N}$, there exists $(g_0,g_1,g_2,g_3)\in \mathcal{G}_{S,T}$ with $ (g_0x_0,g_1x_1,g_2x_2,g_3x_3)$, $(g_0y_0,g_1y_1,g_2y_2,g_3y_3) \in A_i\times A_i \times A_i \times A_i$. By the distality on each coordinate  (and that $\{A_i\}_{i\in \mathbb{N}}$ is a separating sieve), we have that  $(x_0,x_1,x_2,x_3)=(y_0,y_1,y_2,y_3)$ and the claim is proved. 
 
 \
 
 Let the notations be the same as in Subsection \ref{Subsec:Magic}.
 Let $$\mu_{S,T}=\int \mu_{S,T,\vec{x}}d\mu_{S,T}(\vec{x})$$ be the ergodic decomposition of $\mu_{S,T}$ under $\langle{S^{\ast}},{T^{\ast}} \rangle$. It is shown in \cite{DS} that for $\mu_{S,T}$-almost every $\vec{x}\in {X}^4$, the system $(X^{4},\mu_{S,T,\vec{x}}, {S^{\ast}},{T^{\ast}})$ is a free ergodic magic extension of $(X,\mathcal{X},\mu,S,T)$. By Proposition \ref{prop:DistalDecomposition}, for $\mu_{S,T}$-a.e $\vec{x}\in X^4$, $\mu_{S,T,\vec{x}}$ is also distal for $\langle S^{\ast},T^{\ast}\rangle$ and the result follows.
\end{proof}

\subsection{Proof of Theorem \ref{THM:2}}  
In what follows, to lighten notation we use the same letters $S$ and $T$ to denote the transformations in a system and its factors. In order to prove Theorem \ref{THM:2}, it suffices to show that the pointwise convergence of the average $\frac{1}{N} \sum\limits_{i=0}^{N-1} f_1(S^ix)f_2(T^ix)$ can be lifted by some isometric extensions. 
The following result is similar to Theorem 6.1 of \cite{HSY}, but we provide the details for completion: 

\begin{prop}\label{decm}
Let $(X_{1},\mathcal{X}_{1},\mu_{1}, S,T)$ and $(X_{2},\mathcal{X}_{2},\mu_{2}, S,T)$ be two ergodic systems with commuting transformations $S$ and $T$ sharing a common free, ergodic, triple magic extension $(X,\mathcal{X},\mu, S,T)$. Let $p_{i}\colon X\to X_{i}$ be the factor map for $i=1,2$. Then there exist a family $\{\mu_{x}\}_{x\in X}$ of measures on $X_{1}\times X_{2}$ such that 

(1) $\mu_{x}$ is ergodic under $S\times T$ for $\mu$-a.e. $x\in X$.

(2) For all $f_{i}\in L^{\infty}(\mu_{i}), i=1,2$, we have
$$\frac{1}{N}\sum_{n=0}^{N-1}f_{1}(S^{n} p_{1}x)f_{2}(T^{n} p_{2}x)\to \int_{X_{1}\times X_{2}}f_{1}\otimes f_{2} d\mu_{x}$$
in the $L^{2}(\mu)$ norm as $N\to\infty$ for $\mu$-a.e. $x\in X$.
\end{prop}

\begin{proof} 
By Theorem \ref{thm:uniquelyergodicmodel}, we may assume that $X$ is endowed with a topological structure so that $(N_{S,T}(X),H_{S,T})$ is uniquely ergodic with measure $\l_{S,T}$.
Recall from the proof of Theorem \ref{thm:uniquelyergodicmodel} that 
$$\l_{S,T}=\int_{\Z_{S,T}}\theta_{s}\times m_{s}d\nu_{S,T}(s),$$
where 
$$\mu=\int_{\Z_{S,T}}\theta_{s} d\nu_{S,T}(s)$$
is the disintegration of $\mu$ over $\nu_{S,T}$,
and 
$$\mu_{R}=\int_{\Z_{S,T}} m_{s} d\nu_{S,T}(s)$$
is the disintegration of $\mu_{R}$ over $\nu_{S,T}$ (recall that $\mu_{R}=\mu\times_{I(R)}\mu, \mu_{S,T}=\mu_{S}\times_{I(T\times T)}\mu_{S}$). By Lemma \ref{ist}, the $(S\times T)$-invariant $\sigma$-algebra is isomorphic to $\Z_{S,T}$, so $m_s$ is an $(S\times T)$-ergodic measure on $\Q_{R}(X)$ for almost every $s\in \Z_{S,T}$. Therefore, $$m'_{s}\coloneqq(p_{1}\times p_{2})_{*}m_{s}$$ is an $(S\times T)$-ergodic measure on $X_{1}\times X_{2}$ for almost every $s\in \Z_{S,T}$.

Let $\pi_{R}\colon X\to\Z_{S,T}$ be the projection map.
For $x\in X$, let
$\mu'_{x}=m_{\pi_{R}(x)}$ and $\mu_{x}=(p_{1}\times p_{2})_{*}\mu'_{x}.$
Then for $\mu$-a.e. $x\in X$, $\mu_{x}$ is ergodic under $S\times T$. This prove the existence of the family of measures $\{\mu'_x\}_{x\in X}$. We now prove that this family satisfies (2). We first claim that $\lambda_{S,T}=\int \delta_x\times \mu'_x d\mu$. In fact, 

\begin{align*}
\int_{X} \delta_x\times \mu'_xd\mu(x) =\int_{X} \delta_x\times m_{\pi(x)}d\mu(x)=\int_{\Z_{S,T}} \int_{X} \delta_x \times m_{\pi(x)} d\theta_{s} (x) d\nu_{S,T}(s) \\ 
=\int_{\Z_{S,T}} \Bigl(\int_{X} \delta_x d\theta_{s} (x)\Bigr) \times m_{s}  d\nu_{S,T}(s)=\int_{\Z_{S,T}}\theta_{s}\times m_{s}d\nu_{S,T}(s)=\lambda_{S,T}.
\end{align*} 
Fix $f_{i}\in L^{\infty}(\mu_{i}), i=1,2$, and let $g(x)$ be the $L^2$ limit of $\frac{1}{N} \sum_{i=0}^{N-1}f_1(S^i p_{1}x)f_2(T^i p_{2}x)$. By Proposition \ref{Prop:PointConverBigAverage}, for all $f_0 \in L^{\infty}(\mu)$, we have
\begin{align*} 
\int_{X} f_0(x)g(x)d\mu(x) & =\lim_{N\to\infty} \int_{X} \frac{1}{N} \sum_{i=0}^{N-1} f_0(x)f_1(S^i p_{1}x)f_2(T^i p_{2}x)d\mu(x)
\\
&=\lim_{N\to\infty} \int_{X} \frac{1}{N^3} \sum_{i,j,k=0}^{N-1} f_0(S^{j}T^{k} x)f_1(S^{i+j}T^{k} p_{1}x)f_2(S^jT^{i+k} p_{2}x) d\mu(x)
\\&= \int_{X} f_0\otimes (f_1\circ p_{1}) \otimes (f_2\circ p_{2}) d\lambda_{S,T}
\\&=\int_{X} f_0(x)\Bigl(\int_{X\times X} (f_1\circ p_{1}) \otimes (f_2\circ p_{2}) d\mu'_x\Bigr)d\mu(x)
\\&=\int_{X} f_0(x)\Bigl(\int_{X_{1}\times X_{2}} f_1 \otimes f_2 d\mu_x\Bigr)d\mu(x).   
\end{align*}
So $g(x)=\int_{X_{1}\times X_{2}}f_{1}\otimes f_{2} d\mu_{x}$ for $\mu$-a.e. $x\in X$ and the proof is finished.
\end{proof}

\begin{lem}\label{dp}
Let $\pi_i\colon (X_{i},\mathcal{X}_{i}, \mu_{i},S,T)\to (Y_{i},\mathcal{Y}_{i},\nu_{i},S,T)$ be a factor map between two ergodic systems for $i=1,2$. Suppose that there exists a common free, ergodic, triple magic extension system $X$ of $X_{1}$ and $X_{2}$.
Let $\{\mu_{x}\}_{x\in X}$ and $\{\nu_{x}\}_{x\in X}$ be the measures defined in Theorem \ref{decm} (for the couples $X_1,X_2$ and $Y_1,Y_2$). Suppose $\Z_{S,R}(X)$ is a factor of $Y_1$ and $\Z_{T,R}(X)$ is a factor of $Y_{2}$. Then for all $f_{i}\in L^{\infty}(\mu_{i}), i=1,2$, we have
$$\int_{X_{1}\times X_{2}}f_{1}\otimes f_{2} d\mu_{x}=\int_{Y_{1}\times Y_{2} } \mathbb{E}(f_{1}\vert Y_1)\otimes \mathbb{E}(f_{2}\vert Y_2) d\nu_{x}$$
for $\mu$-a.e. $x\in X$.
\end{lem}

\begin{proof}
Let $p_1,p_1',p_2,p_2'$ be the projections from $X$ to $X_1,Y_1,X_2,Y_2$, respectively, and $f_{i}\in L^{\infty}(\mu_{i}), i=1,2$.
By Theorem \ref{Cauchy}-(4) and the fact that $X$ is triple magic, we have that 
$$\lim_{N\to\infty}\Bigl\Vert\frac{1}{N}\sum\limits_{i=0}^{N-1} f_1(S^ip_1x)f_2(T^ip_2x)-\frac{1}{N}\sum\limits_{i=0}^{N-1} \mathbb{E}(f_1\vert\Z_{S,R}(X))(S^ip_1x)\mathbb{E}(f_2\vert\Z_{T,R}(X))(T^ip_2x)\Bigr\Vert_{L^{2}(\mu)}=0.$$ 
 The conditions that $\Z_{S,R}(X)$ is a factor of $Y_1$ and $\Z_{T,R}(X)$ is a factor of $Y_{2}$ allow us to conclude that 
 $$\lim_{N\to\infty}\Bigl\Vert\frac{1}{N}\sum\limits_{i=0}^{N-1} f_1(S^ip_1x)f_2(T^ip_2x)-\frac{1}{N}\sum\limits_{i=0}^{N-1} \mathbb{E}(f_1\vert {Y}_1)(S^ip_1'x)\mathbb{E}(f_2\vert {Y}_2)(T^ip_2'x)\Bigr\Vert_{L^{2}(\mu)}=0.$$ 
 By Theorem \ref{decm}, we have
\[ \int_{X_{1}\times X_{2}}f_{1}\otimes f_{2} d\mu_{x}=\int_{Y_{1}\times Y_2}\mathbb{E}(f_{1}\vert  {Y}_1)\otimes\mathbb{E}(f_2\vert  {Y}_2) d\nu_{x}\]   
for $\mu$-a.e. $x\in X$.
 \end{proof}

\begin{defn}
Let $\pi\colon X\to Y$ be an isometric extension with fiber $H/\Gamma$ and let $\phi\colon H\to \mathbb{R}_{+}$ be a continuous function. We say that $\phi$ is a {\it weight} if $\int_H \phi(h)dm(h) =1$ and $\phi(h^{-1}gh)=\phi(g)$ for all $g,h\in H$.

Let $f\in L^{\infty}(\mu)$. The conditional expectation of $f$ with weight $\phi$ over $Y$ is defined to be 
\[\mathbb{E}_{\phi}(f\vert \mathcal{Y})(x)=\int_{G} f(hx)\phi(h)dm(h). \]
\end{defn} 
\begin{rem}\label{ConditionPhi}
We use the cursive symbol $\mathcal{Y}$ to stress that this function may not be constant on the fibers of $\pi$ (thus is not a function on $Y$). Remark also that if $\phi=1$, $\mathbb{E}_{\phi}(f\vert \mathcal{Y})(x)=\mathbb{E}(f\vert \mathcal{Y})(x)=\mathbb{E}(f\vert {Y})(\pi(x))$.
\end{rem}
These weighted conditional expectations were considered in Proposition 6.3 in \cite{HSY} and they are helpful when lifting the property of pointwise convergence.

\begin{lem} \label{lem:WeightedConditional}
Let $\pi\colon X\to Y$ be an isometric extension with fiber $H/\Gamma$. Let $\phi\colon H\to \mathbb{R}_{+}$ be a weight and $f\in L^{\infty}(\mu)$. Then for $R=S$ or $T$, we have

\[ \mathbb{E}_{\phi}(f\circ R \vert \mathcal{Y})(x)=\int_{H} f\circ h \circ R (x)\phi(h)dm(h).\]

\end{lem}

\begin{proof}
Since $\pi$ is isometric, modulo a measure preserving isomorphism, the dynamics is given by a cocycle $\rho$.  So we may consider $\Phi\colon X \to Y\times H$ as a measure preserving isomorphism so that $\Phi(Sx)=(Sy, \rho(S,y)h'\Gamma)$, $\Phi(Tx)=(Ty, \rho(T,y)h'\Gamma)$, where $\Phi(x)=(y,h'\Gamma)$. The action of the compact group $H$ is given by $\Phi(hx)=(y, h h'\Gamma)$.

Let $\Phi(x)=(y,h'\Gamma)$. We have
\begin{align*}
&\mathbb{E}_{\phi}(f\circ R \vert \mathcal{Y})(x)=\int_{H} f\circ R(hx)\phi(h)dm(h)=\int_{H} f\circ R\circ \Phi^{-1} (y,hh'\Gamma)\phi(h) dm(h)\\
=& \int_{H} f\circ \Phi^{-1} \circ R(y,hh'\Gamma)\phi(h) dm(h)=\int_{H} f\circ \Phi^{-1} (Ry,\rho(R,y)hh'\Gamma)\phi(h) dm(h).  
\end{align*}
Changing variables from $h$ to $\rho(R,y)^{-1}h\rho(R,y)$, and using the invariance of $m$ and $\phi$ under this transformation, we get that 

\begin{align*}
&{E}_{\phi}(f\circ R \vert \mathcal{Y})(x)= \int_{H} f\circ \Phi^{-1} \circ (Ry,h\rho(R,y)h'\Gamma)\phi(h) dm(h)\\
&=\int_{H} f\circ \Phi^{-1}  h R(y,h'\Gamma)\phi(h) dm(h)= \int_{H} f\circ h\circ R(x) \phi(h) dm(h).
\end{align*}
\end{proof}

\begin{prop}\label{liftiso}
Let $(X,\mathcal{X},\mu,S,T)$, $(X_{i},\mathcal{X}_{i},\mu_{i},S,T)$ , $(Y_{i},\mathcal{Y}_{i}, \nu_{i},S,T), i=1,2$ be systems satisfying the assumption in Lemma \ref{dp}. Let $p_{i}\colon X\to X_{i}$, $p'_{i}\colon X\to Y_{i}$ $i=1,2$ be the factor maps. Suppose that $\pi_i\colon (X_{i},\mathcal{X}_{i},\mu_{i},S,T)\to (Y_{i},\mathcal{Y}_{i}, \nu_{i},S,T)$ is an isometric extension with fiber $H_i/\Gamma_i$. 
If the limit
$$\lim_{N\to\infty}\frac{1}{N}\sum_{i=0}^{N-1}f_{1}(S^{i} p'_{1}x)f_{2}(T^{i} p'_{2}x)$$
exists for $\mu$-a.e. $x\in X$ for all $f_{i}\in L^{\infty}(\nu_{i}), i=1,2$, then the limit
$$\lim_{N\to\infty}\frac{1}{N}\sum_{i=0}^{N-1}g_{1}(S^{i} p_{1}x)g_{2}(T^{i} p_{2}x)$$
exists for $\mu$-a.e. $x\in X$ for all $g_{i}\in L^{\infty}(\mu_{i}), i=1,2$.
\end{prop}

\begin{proof}
By Theorem \ref{Weiss}, we may assume that $(X_1,S,T)$, $(X_2,S,T)$, $(Y_1,S,T)$ and $(Y_2,S,T)$ are topological dynamical systems ({\em i.e.} the transformations are continuous) and $\pi_i\colon X_i\to Y_i$, $i=1,2$ is continuous. Note that we cannot assume that the system $X_i$ has the form $Y_i\times_{\rho_i} H_i$. However, there is certainly a measure preserving isomorphism $\Phi\colon X_i\to Y_i\times_{\rho_i} H_i$, which is sufficient for our purposes.

By hypothesis we can find $X'\subseteq X$ with $\mu(X')=1$  such that $$\frac{1}{N} \sum_{i=0}^{N-1} (S^i\times T^i) \delta_{(p_1'x, p_2'x)}$$ converges weakly to $\nu_x \in M(Y_1\times Y_2)$ for all $x\in X'$.   

Let $x\in X'$ and $\mu_x'\in M(X_1\times X_2)$ be any weak limit of $\frac{1}{N} \sum_{i=0}^{N-1} (S^i\times T^i) \delta_{(p_1x, p_2x)}$. Since the transformations $S$ and $T$ are continuous, we have that $\mu_x'$ is invariant under $S\times T$. 

The strategy of the proof is as follows: in the first part, we show that $\mu_x'$ equals to $\mu_x$ (and thus $\frac{1}{N} \sum_{i=0}^{N-1} (S^i\times T^i) \delta_{(p_1x, p_2x)}$ converges weakly to $\mu_x$) in a subset of $X'$ of full measure. Then in the second part, we show that this property allows to lift the pointwise convergence. 

{\bf First part: $\mu_x'=\mu_x$.}

We start with remarking that if $f_{i}$, $f'_{i}\in L^{\infty}(\mu_{i})$ and $\Vert f_{i}\Vert_{L^{\infty}(\mu_{i})},\Vert f'_{i}\Vert_{L^{\infty}(\mu_{i})}\leq 1$ for $i=1,2$, then the telescoping inequality and the Von Neumann Theorem allow us to bound 
 \begin{equation}\label{equa0}
 \begin{split}
 \left \vert  \int_{X_{1}\times X_{2}} \left (f_1\otimes f_2-f'_1\otimes f'_2\right )d\mu'_x\right\vert \leq
 \mathbb{E}(\vert f_1-f'_1\vert\circ p_1\mid \I(S))(x)+\mathbb{E}(\vert f_2-f'_2 \vert \circ p_2\mid \I(T))(x) 
 \end{split}
 \end{equation} 
 for $\mu$-a.e. $x\in X$.

By hypothesis and the continuity of $\pi_1$ and $\pi_2$,
 we have that \[(\pi_1\times\pi_2)_{*}  \mu_x'=\nu_x\] for $\mu$-a.e $x\in X$.    
 We now consider weighted conditional expectations over $Y_i$, $i=1,2$,  given by weights $\phi_i$, $i=1,2$. Let $\mu_{x,\phi_1,\phi_2}\in M(X_1\times X_2)$ be the measure such that

\[ \int\limits_{X_1\times X_2} f_1\otimes f_2 d\mu_{x,\phi_1,\phi_2}=\int\limits_{X_1\times X_2} \mathbb{E}_{\phi_1}(f_1|\mathcal{Y}_1)\otimes \mathbb{E}_{\phi_2}(f_2\vert \mathcal{Y}_2) \mu'_x  \]  
for all $f_{i}\in L^{\infty}(\mu_{i}), i=1,2,$.
By Fubini's Theorem, the invariance of $\mu_{x}'$ under $S\times T$ and Lemma \ref{lem:WeightedConditional}, we have

\begin{align*}
 & \quad\int\limits_{X_1\times X_2}  f_1\circ S\otimes f_2\circ T d\mu_{x,\phi_1,\phi_2} =\int\limits_{X_1\times X_2}\mathbb{E}_{\phi_1}(f_1\circ S\vert \mathcal{Y}_1)\otimes \mathbb{E}_{\phi_2}(f_2\circ T\vert \mathcal{Y}_2)\mu'_x\\
& = \int\limits_{H_1\times H_2} \int\limits_{X_1\times X_2} \Bigl( f_1\circ h_1\circ S \otimes f_2\circ h_2\circ T d\mu'_x \Bigr ) \phi_1(h_1)\phi_2(h_2) dh_1dh_2 \\
& = \int\limits_{X_1\times X_2}\int\limits_{ H_1\times H_2} f_1\circ h_1 \phi_1(h_1)\otimes f_2\circ h_2 \phi_2(h_2)dm_1(h_1)dm_2(h_2)d\mu'_x\\
&= \int\limits_{X_1\times X_2}\mathbb{E}_{\phi_1}(f_1\vert \mathcal{Y}_1)\otimes \mathbb{E}_{\phi_2}(f_2\vert \mathcal{Y}_2)\mu'_x=\int\limits_{X_1\times X_2}  f_1\otimes f_2 d\mu_{x,\phi_1,\phi_2}.
 \end{align*}   
So $\mu_{x,\phi_1,\phi_2}$ is $(S\times T)$-invariant. 
On the other hand, by the fact that $(\pi_1\times \pi_2)_{\ast}\mu_x'=\nu_x$, we have

\begin{align*}
\left \vert \int_{X_1\times X_2}  f_1\otimes f_2 d\mu_{x,\phi_1,\phi_2} \right \vert & \leq \|\phi_1\|_{\infty}\|\phi_2\|_{\infty} \int\limits_{X_1\times X_2\times H_1\times H_2} \vert f_1 \vert \circ h_1\otimes \vert f_2\vert \circ h_2 dm_1(h_1)dm_2(h_2)d\mu'_x \\
& = \|\phi_1\|_{\infty}\|\phi_2\|_{\infty} \int\limits_{X_1\times X_2}  \mathbb{E}(\vert f_1\vert \mid {Y}_{1})\circ \pi_1 \otimes \mathbb{E}(\vert f_2\vert \mid {Y}_{2})\circ \pi_2  d\mu'_x \\
& \leq \|\phi_1\|_{\infty}\|\phi_2\|_{\infty} \int\limits_{Y_1\times Y_2}  \mathbb{E}(|f_1|\mid {Y}_{1})\otimes \mathbb{E}(|f_2|\mid {Y}_{2}) d\nu_x \\
& = \|\phi_1\|_{\infty}\|\phi_2\|_{\infty} \int\limits_{X_1\times X_2}  |f_1|\otimes |f_2| d\mu_x,
\end{align*}
where the last equality follows from Lemma \ref{dp}. So we get that $\mu_{x,\phi_1,\phi_2}$ is also absolutely continuous with respect to $\mu_x$. Since $\mu_x$ is $(S\times T)$-ergodic, we conclude that $\mu_{x,\phi_1,\phi_2}=\mu_x$. 
%Since $\mu_{x,\phi_1,\phi_2}$ is independent of the weight functions, choosing a sequence of weight functions approximating dirac masses, we have that $\mu_{x,\phi_1,\phi_2}\to\mu_x'$. Therefore $\mu_x'=\mu_x$. This is what we show in what follows.

%Let $f_i\in L^{\infty}(\mu_i)$, $i=1,2$ be functions bounded by 1.
%By (\ref{equa0}), we have
%\begin{equation}\label{eq1}
%\begin{split}
%&\quad\int\limits_{X_{1}\times X_{2}} \Bigl\vert f_1\otimes f_2 -f_1\circ h_1 \otimes f_2\circ h_2 \Bigr\vert d\mu'_x 
%\\& \leq   \mathbb{E}(\vert f_1-f_1\circ h_1 \vert \circ p_1\mid \I(S))(x) +  \mathbb{E}(\vert f_2-f_2\circ h_2 \vert \circ p_2\mid \I(T))(x)
%\end{split}
% \end{equation}
%for all $h_{i}\in H_{i}, i=1,2$.

Remember that for $i=1,2$, the topology of $H_i$ is that of weak convergence in measure. This implies that for any $f_i\in L^{\infty}(\mu_i)$ and $\epsilon>0$, if $h_i\in H_i$ is close enough to the identity then $\|f_i-f_i\circ h_i\|_{L^1(\mu_i)}\leq \epsilon$. This fact concerns only the measure and not the topology on $X_i$, $i=1,2$, which gives us the liberty to choose the most suitable topological model for the isometric extension in the beginning of the proof.

Let $\{(f_{1,k},f_{2,k}): k\in \mathbb{N} \}$ be a countable set of continuous functions included and dense in the unit ball of $C(X_{1})\times C(X_{2})$ (with respect to the supremum norm).
 For $k\in \mathbb{N}$, let $B_{k,n}\subseteq H_1\times H_2$ be a ball centered at the origin such that $(h_1,h_2)\in B_{k,n}$ implies that $\|f_{1,k}-f_{1,k}\circ h_1 \|_{L^1(\mu_1)}\leq 2^{-n}$ and $\|f_{2,k}-f_{2,k}\circ h_2 \|_{L^1(\mu_2)}\leq 2^{-n}$. 
{ Let $(\phi_1^{k,n},\phi_2^{k,n})$ be a sequence of pairs weighted functions whose support is included in $B_{k,n}$ (the condition on the support can always be satisfied, we refer to Proposition 6.3 in \cite{HSY}). Define the functions \[ F_{1,k,n}=\int_{H_1} |f_{1,k}-f_{1,k}\circ h_1| \phi_1^{k,n}(h_1) dm_1(h_1) \text{ and } F_{2,k,n}=\int_{H_2} |f_{2,k}-f_{2,k}\circ h_2| \phi_2^{k,n}(h_2) dm_2(h_2)\] and let
  \[F_{k,n}(x)= \mathbb{E}(F_{1,k,n} \mid \I(S))(x)\ + \mathbb{E}(F_{2,k,n} \mid \I(T))(x).\] 

Let $E_{k,n,i}$ denote the set of $x\in X$ such that $F_{k,n}(x)>\frac{1}{i}$. By the Markov inequality, the measure of $E_{k,n,i}$ is at most
\begin{align*} 
i \int_{X} \int_{H_1} |f_{1,k}-f_{1,k}\circ h_1|\circ p_1(x) \phi_1^{k,n}(h_1) dm(h_1) d\mu(x) \\
+ i \int_{X} \int_{H_2} |f_{2,k}-f_{2,k}\circ h_2|\circ p_2(x)\phi_2^{k,n}(h_2)dm(h_2)d\mu(x)  ,
\end{align*} 
which by Fubini's Theorem equals to
\[
i \left( \int_{H_1} \|f_{1,k}-f_{1,k}\circ h_1\|_{L^1(\mu_1)} \phi_1^{k,n}(h_1) dm(h_1) + \int_{H_2} \|f_{2,k}-f_{2,k}\circ h_2\|_{L^1(\mu_2)}\phi_2^{k,n}(h_2)  dm(h_2) \right)
.\]
By definition of $\phi_1^{k,n}$ and $\phi_2^{k,n}$ this last term is  bounded by $i\cdot2^{-n+1}$.
By the Borel-Cantelli Lemma, $$\mu(\limsup_{n} E_{k,n,i})=0.$$ Denote $X''=X'\bigcap_{k,i\in \mathbb{N}}(\limsup_{n} E_{k,n,i})^c$. Then $\mu(X'')=1$. 

On the other hand, we have that

\begin{align*}
&\Bigl \vert \int\limits_{X_{1}\times X_{2}}  f_{1,k}\otimes f_{2,k} d\mu'_x-\int\limits_{X_{1}\times X_{2}}  f_{1,k}\otimes f_{2,k} d\mu'_{x,\phi_1^{k,n},\phi_2^{k,n}} \Bigr\vert \\
= &\Bigl \vert \int\limits_{H_1\times H_2} \Bigl (   \int\limits_{X_{1}\times X_{2}}  f_{1,k}\otimes f_{2,k} d\mu'_x-\int\limits_{X_{1}\times X_{2}}  f_{1,k}\circ h_1\otimes f_{2,k}\circ h_2 d\mu'_{x}\Bigr )\phi_1^{k,n}(h_1)\phi_2^{k,n}(h_2)dm_1(h_1)dm_2(h_2) \Bigr \vert\\
\leq &  \int\limits_{H_1\times H_2} \Bigl ( \int\limits_{X_1\times X_2} \Bigl \vert f_{1,k}\otimes f_{2,k} -f_{1,k}\circ h_1\otimes f_{2,k}\circ h_2 \Bigr \vert d\mu'_{x} \Bigr ) \phi_1^{k,n}(h_1)\phi_2^{k,n}(h_2)dm_1(h_1)dm_2(h_2) \\
\leq &  \int\limits_{X_1\times X_2} \Bigl ( \int\limits_{H_1} \vert f_{1,k}-f_{1,k}\circ h_1\vert\phi_1^{k,n}(h_1) dm_1(h_1) +  \int_{H_2} \vert f_{2,k}-f_{2,k}\circ h_2\vert \phi_2^{k,n}(h_2)dm_2(h_2) \Bigr )d\mu'_{x}  
\end{align*} 

By the Von Neumann Theorem, there exists a subset $X'''\subseteq X''$ of full measure such that for any $n, k\in \mathbb{N}$, the last expression is bounded by 
\begin{align*}
\mathbb{E}(F_{1,k,n} \mid \I(S))(x)\ + \mathbb{E}(F_{2,k,n} \mid \I(T))(x)=F_{k,n}(x).
\end{align*}

Let $x\in X'''$ and $i\in \mathbb{N}$. By the definition of $X'''$, there exists $N\in \mathbb{N}$ such that $F_{k,n}(x)\leq \frac{1}{i}$ for all $n\geq N$.
Since $i$ is arbitrary, we get that for $x\in X'''$, $$\int_{X_{1}\times X_{2}} f_{1,k}\otimes f_{2,k}d\mu'_{x}=\lim_{n\to\infty}\int_{X_{1}\times X_{2}} f_{1,k}\otimes f_{2,k} d\mu'_{x,\phi_1^{k,n},\phi_2^{k,n}}=\int_{X_{1}\times X_{2}} f_{1,k}\otimes f_{2,k}d\mu_{x}.$$ Since $f_{1,k}$ and $f_{2,k}$ are arbitrary in the dense family, we get that the result is also true for all continuous functions bounded by 1 and thus for all continuous functions.  Therefore $\mu'_{x}=\mu_{x}$ and the first part is proved.}

{\bf Second part: lifting the pointwise convergence.}

We now prove that the fact $\mu_{x}'=\mu_x$, $\mu$-a.e. $x\in X$ allows us to lift the pointwise convergence. Let $f_i \in L^{\infty}(\mu_i)$ $i=1,2$ and assume without loss of generality that their $L^{\infty}$ norms are bounded by 1. If $f_{i}\in C(X_{i}), i=1,2$, then by the definition of $\mu'_{x}$, the average
 $$\frac{1}{N}\sum_{i=0}^{N-1}f_1(S^ip_1 x)f_2(T^ip_2 x)$$ converges to $\int f_{1}\otimes f_{2} d\mu_{x}$ for $\mu$-a.e. $x\in X$.
Now let $f_i\in L^{\infty}(\mu_i)$ $i=1,2$ and $\widehat{f}_i\in C(X_i)$ be functions bounded by 1. 
Again by (\ref{equa0}), we have
\begin{equation}\label{equa1}
\begin{split}
\left \vert  \int\limits_{X_{1}\times X_{2}} \left (f_1\otimes f_2-\wh{f_{1}}\otimes \wh{f_{2}}\right )d\mu_x\right\vert \leq
\mathbb{E}(\vert f_1-\wh{f_{1}}\vert\circ p_1\mid \I(S))(x)+\mathbb{E}(\vert f_2-\wh{f_{2}} \vert \circ p_2\mid \I(T))(x).
\end{split}
\end{equation}
By Birkhoff Theorem and the telescoping inequality, we have
\begin{equation}\label{equa2}
\begin{split}
&\quad\limsup_{N\to\infty}\left \vert  \frac{1}{N}\sum_{i=0}^{N-1}f_1(S^ip_1 x)f_2(T^ip_2 x)-\frac{1}{N}\sum_{i=0}^{N-1}\widehat{f}_1(S^ip_1 x)\widehat{f}_2(T^ip_2 x)\right\vert 
\\&\leq
\mathbb{E}(\vert f_1-\wh{f_{1}}\vert\circ p_1\mid \I(S))(x)+\mathbb{E}(\vert f_2-\wh{f_{2}} \vert \circ p_2\mid \I(T))(x)
\end{split}
\end{equation}
for $\mu$-a.e. $x\in X$. Since $\widehat{f}_i\in C(X_i), i=1,2$, we have
\begin{equation}\label{equa3}
\begin{split}
\lim_{N\to\infty}\frac{1}{N}\sum_{i=0}^{N-1}\wh{f_{1}}(S^ip_1 x)\wh{f_{2}}(T^ip_2 x)=\int_{X_{1}\times X_{2}} \wh{f_{1}}\otimes\wh{f_{2}} d\mu_{x}
\end{split}
\end{equation}
for $\mu$-a.e. $x\in X$. Combining (\ref{equa1}), (\ref{equa2}) and (\ref{equa3}), we have
\begin{equation}\label{equa4}
\begin{split}
&\quad\limsup_{N\to\infty}\left \vert\frac{1}{N}\sum_{i=0}^{N-1}f_{1}(S^ip_1 x)f_{2}(T^ip_2 x)-\int_{X_{1}\times X_{2}} f_{1}\otimes f_{2} d\mu_{x}\right \vert
\\&\leq 
2\Bigl(\mathbb{E}(\vert f_1-\wh{f_{1}}\vert\circ p_1\mid \I(S))(x)+\mathbb{E}(\vert f_2-\wh{f_{2}} \vert \circ p_2\mid \I(T))(x)\Bigr)
\end{split}
\end{equation}
for $\mu$-a.e. $x\in X$.
For any $\epsilon>0$, let  \[E_{\epsilon}=\left \{x \in X\colon \limsup_{N\to \infty}\left \vert \frac{1}{N}\sum_{i=0}^{N-1}f_1(S^ip_1 x)f_2(T^ip_2 x)- \int_{X_{1}\times X_{2}} f_{1}\otimes f_{2} d\mu_{x} \right\vert> 2\epsilon \right \}. \] 
By \eqref{equa4}, for a countable dense set of functions $\widehat{f}_i\in C(X_i)$ $i=1,2$, we have    $$\mu(E_{\epsilon})\leq \mu(\{x\colon \mathbb{E}(\vert f_1-\widehat{f}_1\vert\circ p_1\mid \I(S))(x)\geq \epsilon \}) + \mu(\{x: \mathbb{E}(\vert f_2-\widehat{f}_2\vert\circ p_2\mid \I(T))(x)\geq \epsilon \}).$$ 

Now fix $\epsilon >0$ and let $\delta\leq \epsilon$ and $\widehat{f}_i\in C(X_i)$ with $\|\widehat{f}_i-f_i \|_{L^1(\mu_i)}, \leq {\delta^2}$ $i=1,2$. Then Markov Inequality implies that 

\[\mu(E_{\epsilon})\leq \frac{\|\widehat{f}_1-f_1 \|_{L^1(\mu_1)}}{\epsilon}+ \frac{\|\widehat{f}_2-f_2 \|_{L^1(\mu_2)}}{\epsilon} \leq 2\delta.\]

Letting $\delta$ go to zero, we get that $\mu(E_{\epsilon})=0$. Since the set where the pointwise convergence fails is  $\bigcup_{k\in \mathbb{N}}E_{1/k}$, we get the conclusion. 
\end{proof}

\begin{proof}[Proof of Theorem \ref{THM:2}]
Since an ergodic distal system has a distal, free, ergodic, triple magic extension by Proposition \ref{Prop:DistalMagicExtension}, we may assume that $X$ is distal, free, ergodic and triple magic. Since $X$ is distal, the projections $\pi_{T}\colon X\to \Z_{S,R}(X)$ and $\pi_{S}\colon X\to \Z_{T,R}(X)$ are obviously distal. So
there exist a countable ordinal $\eta$ and a directed family of pairs of factors $(X_{\theta,1},\mu_{\theta,1}, S,T)$, $(X_{\theta,2},\mu_{\theta,2}, S,T)$, $\theta\leq\eta$ such that
\begin{itemize}
\item $X_{0,1}=\Z_{S,R}(X)$, $X_{0,2}=\Z_{T,R}(X)$, $X_{\eta,1}=X_{\eta,2}=X$;
\item For $\theta<\eta$, the extension $\pi_{\theta,i}\colon X_{\theta+1,i}\to X_{\theta,i}$ is isometric for $i=1,2$ and is not an isomorphism for at least one of $i=1,2$;
\item For a limit ordinal $\l\leq \eta$, $X_{\l,i}=\lim_{\leftarrow \theta<\l} X_{\theta,i}$, $i=1,2$.
\end{itemize}

Let $X_1$ and $X_2$ be factors of $X$ with factor maps $p_{i}\colon X\to X_{i}, i=1,2$.
We say that the pair $(X_{1},X_{2})$ is \emph{good} if
 the average
 \[\frac{1}{N} \sum_{i=0}^{N-1} f_1(S^{i} p_{1}x)f_2(T^{i} p_{2}x) \]
 converges for $\mu$-a.e. $x\in X$ as $N\to \infty$ for all $f_i\in L^{\infty}(\mu_{i}), i=1,2$. We want to show that $(X,X)$ is good. 
 
 Since all $X_{\theta,1}$ have a common magic extension $X$ and are extensions of 
$\Z_{S,R}(X)$, and all $X_{\theta,2}$ have a common magic extension $X$ and are extensions of 
$\Z_{T,R}(X)$, we conclude from Proposition \ref{liftiso} that if $(X_{\theta,1},X_{\theta,2})$ is good, so is $(X_{\theta+1,1},X_{\theta+1,2})$.

On the other hand, a standard limit argument shows that the property ``good'' is preserved by taking inverse limits. So in order to prove $(X,X)$ is good, it suffices to show that $(\Z_{S,R}(X),\Z_{T,R}(X))$ is good. 

For $W=S,T$ or $R$, let $\pi'_{W}\colon (X,\mu,S,T)\to (X_{W},\nu_{W},S,T)$ be the factor map (recall that $X_{W}$ is the factor of $X$ associated to the $\sigma$-algebra $\mathcal{I}(W)$).
Recall that the systems $(\mathcal{Z}_{S,R}(X),\mu,S,T)$ and $(\mathcal{Z}_{T,R}(X),\mu,S,T)$ are isomorphic to the systems 
$(X_{S}\times X_{R},\nu_{S}\times\nu_{R},\id\times S, T\times T)$ and $(X_{T}\times X_{R},\nu_{T}\times\nu_{R},S\times S, \id \times T)$ respectively by Lemma \ref{2p}. To prove $(\Z_{S,R}(X),\Z_{T,R}(X))$ is good, it suffices to show that  
\begin{equation}\label{tempgood}
\begin{split}
    \frac{1}{N} \sum_{i=0}^{N-1} f_{1}((\id\times S)^{i} (\pi'_{S}x,\pi'_{R}x))f_2((\id \times T)^{i} (\pi'_{T}x,\pi'_{R}x))
\end{split}
\end{equation}
converges for $\mu$-a.e. $x\in X$ as $N\to \infty$ for all $f_1, f_2\in L^{\infty}(\mu)$. By a density argument, we may assume that $f_{1}=g_{1}\otimes h_{1}, f_{2}=g_{2}\otimes h_{2}$, where $g_{1}\in L^{\infty}(\nu_{S}),g_{2}\in L^{\infty}(\nu_{T}),h_{1},h_{2}\in L^{\infty}(\nu_{R})$. In this case, (\ref{tempgood}) equals to
\begin{equation}\nonumber
\begin{split}
 g_{1}(\pi'_{S}x)g_{2}(\pi'_{T}x)\frac{1}{N}\sum_{i=0}^{N-1} h_{1}( S^{i} \pi'_{R}x)h_2(T^{i} \pi'_{R}x).
\end{split}
\end{equation}
Since $S$ and $T$ are the same action on $X_R$, the Birkhoff Theorem implies that $(\Z_{S,R}(X),\Z_{T,R}(X))$ is good. This finishes the proof.  
\end{proof}

\end{document}